\documentclass[12pt]{article}
\pdfoutput=1

\usepackage{color}
\usepackage{amssymb,amsthm,amscd,mathrsfs,amsbsy,amsfonts}
\usepackage{pstricks,pst-node}
\usepackage{amsbsy,young,colortbl,cite}

\usepackage{makecell}
\usepackage{longtable}
\usepackage{threeparttable}
\usepackage{supertabular}
\usepackage{caption}
\usepackage{graphicx}
\graphicspath{{figures/}} 
\usepackage[justification=centering]{caption}
\usepackage{subfigure}
\usepackage{float}  

\usepackage{amssymb,amsthm}
\usepackage[english]{babel}

\usepackage{algorithm}
\usepackage{algorithmic}
\usepackage{multirow} 

\usepackage{threeparttable}

\usepackage{hyperref}
\hypersetup{hypertex=true,colorlinks=true,linkcolor=blue,citecolor=green}

\newtheorem{proposition}{Proposition}
\newtheorem{lemma}{Lemma}
\newtheorem{definition}{Definition}
\newtheorem{theorem}{Theorem}

\newtheorem{remark}{Remark}
\newtheorem{example}{Example}

\renewcommand{\L}{\mathcal L}

\newcommand{\M}{\mathcal M}

\newcommand{\R}{\mathbb R}

\newcommand{\N}{\mathbb N}

\renewcommand{\L}{\mathcal{L}}

\newcommand{\x}{\mathbf{x}}

\def\({\left(}
\def\){\right)}
\def\[{\begin{eqnarray}}
\def\]{\end{eqnarray}}

\usepackage{amsmath}
\numberwithin{equation}{section}

\usepackage{amssymb}

\begin{document}

\title{Similarity Between Two Dynamical Systems\footnotemark[1]}
\author{\  \  Xiaoying Wang \footnotemark[2] , Yong Li \footnotemark[2] \footnotemark[3] , Yuecai Han \footnotemark[2]}
\date{}
\maketitle
\renewcommand{\thefootnote}{\fnsymbol{footnote}}
\footnotetext[1]{\textbf{Funding:} The work of  Yong Li is partially supported by National Natural Science Foundation of China (No. 12071175, 11901056, 11571065),
Jilin Science and Technology Development Program (No. 20190201302JC, 20180101220JC).
The work of Yuecai Han is partially supported by National Natural Science Foundation of China (No. 11871244) and Jilin Science and Technology Development Program (No. 20190201302JC).}

\footnotetext[2]{School of Mathematics, Jilin University, Changchun 130012, P. R. China. \\ xiaoying21@mails.jlu.edu.cn, liyong@jlu.edu.cn, hanyc@jlu.edu.cn (Corresponding author).}
\footnotetext[3]{School of Mathematics and Statistics, Center for Mathematics and Interdisciplinary Sciences, Northeast Normal University, Changchun 130024, P. R. China.}

\renewcommand{\thefootnote}{\arabic{footnote}}

{\bf Abstract.}
The main focus of this paper is to explore how much similarity between two dynamical systems.
Analogous to the classical Hartman-Grobman theorem, the relationship between two systems can be linked by a homeomorphic map $K$, and the core is to study the minimizer $K^*$ to measure the degree of similarity.
We prove the sufficient conditions and necessary conditions (the maximum principle) for the existence of the minimizer $K^*$.
Further, we establish similarity theorem based on the Takens embedding theorem. As applications, Lorenz system, Chua's circuit system and Chen's system are simulated and tested. The results illustrate what is the similarity, which extends the conjugacy in dynamical systems. 

{\bf Keywords.} Similarity, conjugacy, maximum principle, Takens embedding theorem.
\allowdisplaybreaks
\section{Introduction}
\ \ \ \ Many processes or phenomena in nature and society are similar in some characteristics. How to find or extract them quantitatively is a challenging problem. For example, in dynamic systems, if two systems are conjugate, they will always be similar (qualitative behavior). In this regard, some important work has been done.
In addition to the classical papers of Hartman \cite{[H+60]} and Grobman \cite{[G+65]}, we can find the case of finite dimensional cases in the monographs of Hale \cite{[WL+12]}. So far, we have known some specific versions of Hartman-Grobman theorem, such as parabolic evolution equations (e.g. scalar reaction-diffusion equations \cite{[L+91]}, Cahn-Hilliard equation and similar phase-field equations \cite{[BL+94]}), the hyperbolic evolution equations (e.g. semilinear \cite{[HP+16]}, nonuniform \cite{[BV+06],[BV+09]}), control systems \cite{[BCP+07]}, and so on. For the smoothness of the conjugacy in the Hartman-Grobman theorem, see \cite{[ZLL+22]} and \cite{[ZLZ+17]}.

In this paper, we re-examine the dynamical systems from the perspective of similarity.
If we have given two systems described by differential equations respectively, what is their similarity? How can this similarity be determined by means of cost functional? We will touch these problems.

To be more precise, let $T>0$ be a fixed time duration and $D\subset\R^n$ be a bounded closed domain. For all $t\in[0, T]$, we consider the similarity between the following two dynamic systems, which are two ordinary differential equations (ODEs) that will run through the whole article:
\begin{equation}\label{2nonlinear}
	\left\{ \begin{aligned}
		\overset{.}x(t) &=f(t,x(t)),\\
		          x(0)&=x_0, \\ \end{aligned}  \right.~
	\left\{  \begin{aligned}
		\overset{.}y(t) &=g(t,y(t)),\\
	             y(0)&=y_0, \\  \end{aligned}  \right.
\end{equation}
where $f(t,x(t))$, $g(t,y(t)):[0,T]\times \mathbb{R}^n\rightarrow\mathbb{R}^n$ are functions; $x_0$, $y_0\in D$ are initial values, and $x(t),y(t)\in C^1([0,T]; \mathbb{R}^{n})$ are the strong solutions \footnote{We assume that the pair solution $(x,y)$ is guaranteed to exist and be unique for given mild conditions: $f(t,x(t))$, $g(t,y(t))$ are continuous with respect to $t$ and locally Lipschitz with respect to $x$ or $y$.} of (\ref{2nonlinear}). We denote by $x=x(t,x_0)$, $y=y(t,y_0)$ the solution of the first and second equation respectively.

As a special case of (\ref{2nonlinear}): A linear system
\begin{equation}\label{Hartman x}\left\{\begin{aligned}
\dot{x}(t)&=Ax(t),\\
x(0)&=x_0,
\end{aligned}\right.
\end{equation}
and a nonlinear system
\begin{equation}\label{Hartman y}\left\{\begin{aligned}
\dot{y}(t)&=Ay(t)+r(y(t)),\\
y(0)&=y_0,
\end{aligned}\right.
\end{equation}
where $A\in \R^{n\times n}$ is a constant matrix and $r(y)=o(|y|)\in C^1(\R^n)$, $|y(t)|=\underset{i=1}{\overset{n}{\sum}}|y_i(t)|$,
can be linked by the classical Hartman-Grobman theorem (linearisation theorem):
Nearby the hyperbolic \footnote{Here hyperbolic means that the matrix $A$ has no purely imaginary eigenvalues.} equilibrium
$x^*=\mathbf{0}$ \footnote{In fact, we can shift the system to the origin $O$ and consider the conjugate problem. More precisely, let $z(t,z_0)=y(t,y_0)-y_0$, then $\dot{z}(t)=\dot{y}(t)=g(t,y(t))=g(t,z(t)+y_0)\triangleq \tilde{g}(t,z(t))$ and $z(\mathbf{0})=z_0=\mathbf{0}$, where $\mathbf{0}$ is zero vector.}, the behaviour of the nonlinear system (\ref{Hartman y}) is completely similar to the linear system (\ref{Hartman x}) by means of a hemeomorphic coordinate transformation.

Inspired by the Hartman-Grobman theorem, we mainly study the existence and uniqueness of the homeomorphism (bijective and bicontinuous) $K$ for general nonlinear systems (\ref{2nonlinear}).
For the sufficient conditions of existence, we discuss the compact (Theorem \ref{Polynomial space's minimizer} and \ref{exists in GL(n)}) and noncompact (Theorem \ref{Approximate minimum}) spaces composed of homeomorphic maps. For the necessary conditions of existence, we prove the maximum principle of similarity between the two systems (Theorem \ref{max1}).

If there exists a homeomorphic map $K$ that maps $x$ to $y$ in (\ref{2nonlinear}),
then studying the system $y$ is equivalent to studying the system $x$ and the form of $K$.
However, even in the finite-dimensional case, 
the homeomorphism $K$ which linearises the nonlinear problem in line  with
\begin{equation}\label{Kx=x}
K(e^{At}x_0)=y(t,K(x_0)),t\in[0,T],x_0\in D,
\end{equation}
is in general not unique, let alone write its form explicitly.

In order to explicitly represent $K$, we have done some work. For the sufficient conditions of the existence, if we make $K$ explicitly change with $t$, we can explicitly write the form of $K$, i.e. the $K(t)$ in Theorem \ref{exists in GL(n)}. In addition, according to the controllability theory of systems, when the time terminal $t_1$ is fixed, we hope to obtain the terminal $y(t_1)$ through $x$. We can explicitly write the form of the homeomorphism $K$ of (\ref{Hartman y}) and (\ref{Hartman x}), i.e. the $K(t_1)$ in Theorem \ref{terminal K}. At this time, $K$ is independent of time $t$.

Then we propose the concept of the similarity degree of any two systems by referring to the conjugacy theory of dynamical systems, and establish the general principles of the similarity by introducing suitable transformation. 

Our theoretical and numerical results show that the similarity between two differential dynamical systems and the minimizer $K^*$ of cost functional are intimately related.
It means that we can connect two seemingly unrelated systems through some similarity transformation, make prediction based on the known data, and realize the integration and unity of the systems.

The remainder of the paper is organized as follows: In Section \ref{DL}, we put forward some concepts (conjugacy, similarity, etc.), in which the minimizer $K^*$ is the key to the similarity between two systems. 
For the existence of the $K^*$ minimizer, the sufficient condition is proved theoretically in Section \ref{Sufficient}, while the necessary condition (the maximum principle) is proved theoretically in Section \ref{MAX}.
Section \ref{Uniqueness Similarity} shows the uniqueness of the minimizer and the similarity between different systems in both continuous form and discrete form. In Section \ref{embed}, we propose the similarity theory based on Takens embedding theorem. Finally, we provide some numerical simulation between two nonlinear systems (Lorenz system, Chua's circuit system and Chen's system) in Section \ref{test}.

\section{Preliminaries}\label{DL}
\ \ \ \ The homeomorphism $K$ in (\ref{Kx=x}) characterizes topological conjugation. Based on this, we give the following definitions characterizing the relationship between two systems $x$ and $y$.

Let $\Gamma_{x_0}:=\{\rho\in\R^n|\rho=x(t,x_0),~t\in[0,T]\}$ represent the trajectory of $x$, and $\Gamma_{y_0}:=\{\tilde{\rho}\in\R^n|\tilde{\rho}=y(t,y_0),~t\in[0,T]\}$ represent the trajectory of $y$. Let $\Omega:=\{K:\Gamma_{x_0}\rightarrow\Gamma_{y_0}|K\text{ is homeomorphic}\}$
denote the admissible set. Let $N(x_0,\delta)=\{x_0^+\in D\mid|x_0^+-x_0|<\delta\}$.

\begin{definition}\label{conjugation}
\textbf{(System Conjugacy)}
Suppose that $x(t)$, $y(t)$ satisfy (\ref{2nonlinear}).
If there are a homeomorphic map $K:\Gamma_{x_0}\rightarrow\Gamma_{y_0}$
and a constant $\delta\geq0$, 
for all $x_0^+\in N(x_0,\delta)$ such that
\begin{equation*}
    y_0^+=K(x_0^+),~K(x(t,x_0^+))=y(t,y_0^+),~t\in[0,T],
\end{equation*}
then we call systems $x$ and $y$ locally conjugate in $N(x_0,\delta)$.
Further, if there admits a homeomorphic map $K:\Gamma_{x_0}\rightarrow\Gamma_{y_0}$, for all $x_0^+\in D$ such that
\begin{equation*}
    y_0^+=K(x_0^+),~K(x(t,x_0^+))=y(t,y_0^+),~t\in[0,T],
\end{equation*}
then we call systems $x$ and $y$ conjugate.

\end{definition}

\begin{definition}\label{semi-conjugation}
\textbf{(System Semi-conjugacy)}
Suppose that $x(t)$, $y(t)$ satisfy (\ref{2nonlinear}).
If there are two homeomorphic maps $K, R:\Gamma_{x_0}\rightarrow\Gamma_{y_0}$
and a constant $\delta\geq0$, 
for all $x_0^+\in N(x_0,\delta)$ such that
\begin{equation*}
    y_0^+=K(x_0^+),~R(x(t,x_0^+))=y(t,y_0^+),~t\in[0,T],
\end{equation*}
then we call systems $x$ and $y$ locally semi-conjugate in $N(x_0,\delta)$.
Further, if there admit two homeomorphic maps
$K, R:\Gamma_{x_0}\rightarrow\Gamma_{y_0}$, for all $x_0^+\in D$ such that
\begin{equation*}
    y_0^+=K(x_0^+),~R(x(t,x_0^+))=y(t,y_0^+),~t\in[0,T],
\end{equation*}
then we call systems $x$ and $y$ semi-conjugate.

\end{definition}
There are some elementary examples in Section \ref{egs}, which help us have a specific understanding of the homeomorphic map $K$ in conjugate case.\\

Sometimes, there is no such conjugacy or semi-conjugacy between two systems. 
So we go back to the second place and find a certain degree of conjugacy, and this is why we propose the definition of cost functional (Definition \ref{Cost functional}).
We define the following functionals, which are the extensions of Definition \ref{conjugation} and Definition \ref{semi-conjugation}:
\begin{definition}\label{Cost functional}
\textbf{(Cost functional)} Suppose that $x(t)$, $y(t)$ satisfy (\ref{2nonlinear}). Follow the definitions and notations in Definition \ref{conjugation} and Definition \ref{semi-conjugation} \footnote{$||\cdots||$ represents the norm of vector, such as the $L^1$ norm $||x||=|x_1|+|x_2|+\cdots+|x_n|$ and the $L^2$ norm (Euclidean norm) $||x||=\sqrt{x_1^2+x_2^2+\cdots+x_n^2}$, etc.. They are all equivalent. }
\begin{equation*}
	\tilde{J}[K]\triangleq\underset{x_0^+\in N(x_0,\delta)}{\sup}\dfrac{1}{T} \int_{0}^{T}||K(x(t,x_0^+))-y(t,K(x_0^+))||^2 dt ,
\end{equation*}
\begin{equation*}
	\tilde{J}[K,R]\triangleq\underset{x_0^+\in N(x_0,\delta)}{\sup}\dfrac{1}{T} \int_{0}^{T}||R(x(t,x_0^+))-y(t,K(x_0^+))||^2 dt .
\end{equation*}
\end{definition}

For each fixed initial value $x_0$, the trajectory $\Gamma_{x_0}$ is determined. Without losing generality, we assume that the functional $\tilde{J}[\cdot]$ take the maximum value at $x(t,x_0)$, that is:
\begin{equation}\label{conjugation functional}
	\tilde{J}[K]=\dfrac{1}{T} \int_{0}^{T}||K(x(t,x_0))-y(t,K(x_0))||^2 dt \triangleq J[K],
\end{equation}
\begin{equation*}\label{semi-conjugation functional}
	\tilde{J}[K,R]=\dfrac{1}{T} \int_{0}^{T}||R(x(t,x_0))-y(t,K(x_0))||^2 dt \triangleq J[K,R],
\end{equation*}
where $K,R \in \Omega$. 
$J[K],J[K,R]$ are the lower semi-continuous functions, and obviously, $J[K]\geq0,~J[K,R]\geq0$.

If the functional $J[K]$ exists a minimum and $J[K^*]$ is the minimum,
$$J[K^*]=\underset{K\in\Omega}{\inf}~\underset{x_0^+\in N(x_0,\delta)}{\sup}\dfrac{1}{T} \int_{0}^{T}||K(x(t,x_0^+))-y(t,y_0^+)||^2 dt,$$
then $x$ and
$y$ can satisfy the conjugacy to a certain extent. Particularly, if $J[K^*]=0$, then $||K^*(x(t,x_0 ))- y(t,K^*(x_0))||^2=0$, $K^*(x(t,x_0 ))=y(t,K^*(x_0 ))$, thus $x$ and $y$ conjugate.



Analogously, if the functional $J[K,R]$ exists a minimum and $J[K^*,R^*]$ is the minimum, then $x$
and $y$ can satisfy the semi-conjugacy to a certain extent. Particularly, if
$K=R$, then J$[K,R]$ is reduced to $J[K]$.\\

Clearly, the larger the cost functional (Definition \ref{Cost functional}), the smaller the similarity between the two dynamics,
and the range of $J[K],J[K,R]$ is $[0,+\infty]$. 

When $J[K]=0~(J[K,R]=0)$, the two systems conjugate (semi-conjugate). In other words, they are completely similar (semi-similar), the corresponding similarity degree (semi-similarity degree) should be 1. When $J[K]=+\infty~(J[K,R]=+\infty$), the two systems are completely dissimilar, the corresponding similarity degree (semi-similarity degree) should be 0.

It is worth mentioning that these are independent of the selection of the similarity degree function. Based on this, we give the concept of the similarity degree function to quantitatively describe the similarity between two systems.


\begin{definition}\label{similarity}
\textbf{(Similarity degree)} \label{similarity0}
Let the function $\rho(x)$ be continuous and monotonically decreasing in $[0,+\infty]$ such that 
$\rho(0)=1$ and $\rho(+\infty)=0$, then we call $\rho(x)$ a similarity degree function.

We call $\rho(J[K])$ the similarity degree of the systems (\ref{2nonlinear}) with respect to $\rho$ and $\rho(J[K,R])$ the semi-similarity degree of the systems (\ref{2nonlinear}) with respect to $\rho$, respectively.
In particular, when conjugating or semi-conjugating, the corresponding similarity and semi-similarity hold with $\rho(J[K^*])=1,~\rho(J[K^*,R^*])=1$.
\end{definition}

For example, $\rho(x)=\frac{\log(1+x)}{x}$ \footnote{Notice that $\underset{x\rightarrow0^+}{\lim}\frac{\log(1+x)}{x}=1$ and $\underset{x\rightarrow +\infty}{\lim}\frac{\log(1+x)}{x}=0$.}, then
\begin{equation}
	\rho(J[K])=\left\{ \begin{aligned}	\frac{\log\Big(1+\underset{K\in\{\Omega|K(x_0)=y_0\}}{\inf}~J[K]\Big)}{\underset{K\in\{\Omega|K(x_0)=y_0\}}{\inf}~J[K]}, &\qquad \underset{K\in\{\Omega|K(x_0)=y_0\}}{\inf}~J[K]\neq0,\\
		          1,\qquad\qquad &\qquad \underset{K\in\{\Omega|K(x_0)=y_0\}}{\inf}~J[K]=0, \\ \end{aligned}  \right.
\end{equation}
\begin{equation*}
\rho(J[K,R])=\left\{ \begin{aligned}
		\frac{\log\Big(1+\underset{(K,R)\in\{\Omega\times\Omega|K(x_0)=y_0\}}{\inf}~J[K,R]\Big)}{\underset{(K,R)\in\{\Omega\times\Omega|K(x_0)=y_0\}}{\inf}~J[K,R]}, &\qquad \underset{(K,R)\in\{\Omega\times\Omega|K(x_0)=y_0\}}{\inf}~J[K,R]\neq0,\\
		          1,\qquad\qquad &\qquad \underset{(K,R)\in\{\Omega\times\Omega|K(x_0)=y_0\}}{\inf}~J[K,R]=0. \\ \end{aligned}  \right.
\end{equation*}


The following discussion is mainly for $J[K]$ in the case of a homeomorphism map $K$, and the proof of $J[K,R]$ is completely analogous.

The core problem is to find \textbf{the minimizer $K^*$}, 
\textbf{which decides the similarity between two systems.} If $K^*(\cdot)$ is a constant matrix, then the similarity is the linear similarity. If $K^*(\cdot)$ is orthogonal (metric preserving) or symplectic (differential structure preserving), then the similarity is called the rigid similarity.

\subsection{Analogous to Hartman-Grobman theorem}
\ \ \ \ Inspired by the classical Hartman-Grobman theorem, we consider the conjugate functional problem (\ref{conjugation functional}) of equation (\ref{Hartman y}) and equation (\ref{Hartman x}).
In this case, the form of the minimizer $K^*$ can be found.

We recall the definition of the dichotomic projection and the associated Green kernel of the generator $A$ of the $C_0$-semigroup $e^{At}$ in a Banach space $X$.
\begin{definition}
A projection $P_+\in\mathcal{B}(X)$ is called a dichotomic projection or (exponential) dichotomy for the $C_0$-semigroup $e^{At}$ or of its generator $A$ in $X$, if there are constants $M\geq1,\eta>0$
such that with $P_-=I_X-P_+$ the following conditions are satisfied:
$$
\begin{aligned}
&(i)~P_+e^{At}=e^{At}P_+,\forall~t\geq0,\\
&(ii)~|e^{At}P_+x|_X\leq Me^{-\eta t}|P_+x|_X,\forall~t\geq0,\forall~x\in X,\\
&(iii)~e^{At}P_-\text{ extends to  a } C_0-\text{group on }R(P_-),\\
&(iv)~|e^{At}P_-x|_X\leq Me^{\eta t}|P_-x|_X,\forall~t\leq0,\forall~x\in X.
\end{aligned}
$$
The Green kernel corresponding to the dichotomy is defined as
$$
G_A(t)=\left\{\begin{aligned}
e^{At}P_+,t\geq0,\\
-e^{At}P_-,t<0.
\end{aligned}\right.
$$
\end{definition}

As the results in \cite{[HP+16]}, the following proposition is obtained.
\begin{proposition}\label{Hartman}
Suppose that $x(t)$, $y(t)$ satisfy equation (\ref{Hartman x}) and equation (\ref{Hartman y}). Let $A$ admit a dichotomic projection and $e^{At}$ represent the fundamental solution matrix of equation (\ref{Hartman x}) with Green kernel $|G_A(t)|\leq Me^{-\eta|t|},t\in\R$, with some constants $M\geq1$ and $\eta>0$. Let $r(\cdot)$ be bounded and Lipschitz continuous ,with the Lipschitz constant $|r|_{Lip}$ of the map $r(\cdot)$ satisfies \begin{equation}\label{r Lipschitz}
(\frac{2M}{\eta})|r|_{Lip}<1.
\end{equation}
Then there is a homeomorphism $K^*=I_X+g$, such that the conjugate functional (\ref{conjugation functional}) is equal to 0 (i.e. $J[K^*]=0$), where $I_X$ is the identity map and $g(x)=\int_\R G_A(s)r(e^{-As}x+g(e^{-As}x))ds$.
\end{proposition}
\begin{proof}
The $y(t)=e^{At}K(x_0)+\int_0^te^{A(t-s)}r(K(e^{As}x_0))$ is the solution of (\ref{Hartman y}) with initial value $K(x_0)$, then the functional (\ref{conjugation functional}) is
$$
J[K]=\frac{1}{T}\int_0^T||K(e^{At}x_0)-[e^{At}K(x_0)+\int_0^te^{A(t-s)}r(K(e^{As}x_0))]||^2ds.
$$

Actually, the homeomorphism map is $K^*(x)=x+g(x)$ (i.e. $K^*=I_X+g$) such that (\ref{Hartman y}) and (\ref{Hartman x}) conjugate, where $I_X$ is the identity map and $g(x)=\int_\R G_A(s)r(e^{-As}x+g(e^{-As}x))ds\triangleq (Tg)(x)$.
Therefore, we only need to prove the fixed point problem of the map $T$ via the Banach contraction mapping principle. 

We calculate
$$|Tg|_\infty\leq\int_\R |G_A(s)|~|r|_\infty ds\leq2M|r|_\infty\int_0^{+\infty}e^{-\eta s}ds=\frac{2M}{\eta}|r|_\infty<\infty,$$
and
\begin{equation}\begin{aligned}
|Tg(x)-T\bar{g}(x)|_\infty&\leq\int_\R |G_A(s)|~|r|_{Lip}~|g(e^{-As}x)-\bar{g}(e^{-As}x)| ds\\
&\leq2M|r|_{Lip}~|g-\bar{g}|_\infty\int_0^{+\infty}e^{-\eta s}ds\\
&=\frac{2M}{\eta}|r|_{Lip}~|g-\bar{g}|_\infty.
\end{aligned}\end{equation}
The condition (\ref{r Lipschitz}) shows that $T$ is a contractive map, so there is a unique fixed point. This implies the homeomorphism $K^*$ exists with the form $K^*=I_X+g$, and is unique.
Using the properties of the Green kernel $G_A$, $K^*$ is given by
\begin{equation}
K^*(e^{At}x_0)=e^{At}K^*(x_0)+\int_0^te^{A(t-s)}r(K^*(e^{As}x_0))ds=y(t,K^*(x_0)),\forall~t\in[0,T].
\end{equation}
Since (\ref{Hartman y}) and (\ref{Hartman x}) conjugate, $J [K^*]=0$.
\end{proof}

In general, two nonlinear systems (\ref{2nonlinear}) are not necessarily conjugate. In order to weaken the conjugacy between $x$ and $y$, we do not require the conjugate equation to be strictly true. Instead, we look for the conditions for the existence of functional minimum in (\ref{conjugation functional}). We will prove the sufficient conditions for the existence of the minimizer $K^*$ in Section \ref{Sufficient}, and the necessary conditions for the existence of the minimizer $K^*$ in Section \ref{MAX}. The main ideas are as follows.
\subsection{The main idea}\label{The main idea}
\ \ \ \ Based on the conditions we mentioned above, 
norm and distance in $\Omega$ are defined respectively as follows: $||K||_\Omega\triangleq\underset{x\in\Gamma_{x_0}}{\sup}~|K(x)|,~d(K_1,K_2)\triangleq||K_1-K_2||_\Omega$. Then $\Omega,~\Omega\times\Omega$ are normed spaces. To explain our approach, we begin with some motivations by recalling some useful facts to be applied below. \\

We first consider the (sufficient) existence of the minimizer $K^*$ of functional $J[K]$, where $K\in\Omega\subset C[\Gamma_{x_0};\Gamma_{y_0}]$.

\textbf{(i)} From the previous definition, $\Gamma_{x_0},\Gamma_{y_0}\subset[0,T]\times\R^n$ are two bounded closed domains. According to the topological method (Poincar\'{e} Theorem \cite{[Poincare]}), any two curve segments are topologically homeomorphic, which means  they are contained in homeomorphic two closed spheres. Then, the homeomorphism $K$ is bounded. Therefore, the minimization sequence is bounded, and in the sense of $\mathcal{L}^2$, the weak convergence subsequence can be taken.

\textbf{(ii)} Take the weakly convergent subsequence (in the sense of functional), and it can be regarded as the approximate minimum of functional. The approximate minimum of functional can be obtained by Ekeland variational principle (Lemma \ref{Ekeland variational principle}). Therefore, the approximate minimization sequence is bounded. Here, approximation means that the obtained sequence is related to $\epsilon$ (see Section \ref{K in Omega} Theorem \ref{Approximate minimum} for details).
\begin{lemma}\label{Ekeland variational principle}
\textbf{(Ekeland variational principle)}
Let $(V,d)$ be a complete metric space, $F:V\rightarrow\R\cup\{+\infty\}$ be a lower-semi-continuous function on $V$, and $\underset{v\in V}{\inf}~F(v)<+\infty$. Fix $\epsilon>0$, if there exists $u\in V$ s.t. $F(u)\leq\underset{v\in V}{\inf}~F(v)+\epsilon$, then for every $\lambda>0$, there exists a point $u^\epsilon\in V$ such that
$$\begin{cases}
  i)~F(u^\epsilon)\leq F(u),\\
  ii)~d(u^\epsilon,u)\leq\lambda,\\
  iii)~F(u^\epsilon)\leq F(w)+\frac{\epsilon}{\lambda}d(w,u^\epsilon),~\forall ~w\in V.
\end{cases}
$$
\end{lemma}
\textbf{(iii)} The minimization sequence generally cannot converge strongly. Therefore, using the vector field of gradient flow in Hilbert space, the strong limit cannot be taken. And the limit point of weak convergence may not be unique. Hence, it cannot be concluded that the field at the limit point is non-zero.

However, if the space is compact, the strongly convergent subsequence can be obtained according to Arzel\`{a}-Ascoli lemma. We will discuss the minimizer $K^*$ in the Polynomial space (Section \ref{K in Omega} Theorem \ref{Polynomial space's minimizer}).

Note that for (\ref{Hartman y}) and (\ref{Hartman x}), the homeomorphism $K^*=I_X+g\in\Omega$ in Proposition \ref{Hartman} is not explicitly represented by $t$, but implicitly represented by $x (t)$, where $g(x)=\int_\R G_A(s)r(e^{-As}x+g(e^{-As}x))ds$.
For general nonlinear systems (\ref{2nonlinear}), we cannot obtain such conclusion directly.
Instead, we consider explicitly expressing $K^*$ as $K^*(t)$ and discuss the form when $K^*(t)\in GL(n)\cap\Omega$ \footnote{$GL (n)$ represents a general linear matrix group, and the matrices in this group are of order $n$.} (Section \ref{K(t) in GL(n)} Theorem \ref{exists in GL(n)}).

Further, we consider $x(t_1)=x_1,y(t_1)=y_1$ on the trajectories $x(t),y(t)$ when the terminal $t=t_1\in[0,T]$ is given. We will discuss the form of $K^*$ when $K^*x_1=y_1$ in Section \ref{When K in H}.\\

Then we consider how to determine the minimizer $K^*(t)$, which is the necessary condition for existence, like Pontryagin maximum principle.

\textbf{(iv)} If the strong limit exists, the variation at the extreme point must be 0. The following lemma converts the variation of a functional into the derivation of a function.
\begin{lemma}
If functional $J[K(\cdot)]$ has variation $\delta J[\delta K(\cdot)]$
\footnote{The variation $\delta J[\delta K(\cdot)]$ means that the functional $J$ changes $\delta K(\cdot)$ at $K_0$, and the corresponding dependent variable $\Delta J\triangleq J[K_0(\cdot)+\delta K(\cdot)]-J[K_0(\cdot)]$ satisfies $\Delta J=\delta J[\delta K(\cdot)]+o(\parallel \delta K(\cdot)\parallel)$.} at $K_0$, $\forall ~\epsilon \in \mathbb{R}$, 
let
\begin{equation*}
	J(\epsilon)=J[K_0 (\cdot)+\epsilon \delta K(\cdot)],
\end{equation*}
then
$$
\begin{cases}
i)~J(\epsilon)\text{ is defined near } \epsilon=0 \text{ and differentiable at }\epsilon=0,\\
ii)~\delta J[\delta K(\cdot)]=\dfrac{d}{d\epsilon} J(\epsilon)|_{\epsilon =0}.
\end{cases}
$$
\end{lemma}

Besides, we will use the fundamental variational principle (Lemma \ref{L2}) to derive the maximum principle (Section \ref{MAX}).
\begin{lemma}\label{L2}
Let $\xi (t)$ be a $k$-dimensional piecewise continuous vector valued function
defined on $[t_0,t_1 ]$. If for any $k$-dimensional piecewise continuous vector
valued function $\eta (t)$ defined on $[t_0,t_1 ]$, we have
\begin{equation*}
	\int_{t_0}^{t_1}\xi^T(t)\eta (t)dt=0,
\end{equation*}
then at all successive moments $\xi (t)=0$.
\end{lemma}
\textbf{(v)} 
Through the chain rule of derivatives, we can solve for $K(t)$ and obtain its uniqueness (Section \ref{Uniqueness}).

If $K\in\mathcal{M}^n$ \footnote{$\M^n$ represents the $n$-order constant matrix.}, the minimizer $K^*$ is a constant matrix, two systems are linearly similar. Intuitively, when taking photos, constant matrix $K$ means that we only need to convert one camera to make the two photos linear similar. $K(t)$ means that we need to constantly change the camera over time and constantly carry out translation, rotation, expansion and contraction transformation, so as to make the two photos completely similar.

We propose the similarity theory for prediction based on Takens embedding theorem \cite{[Takens+81]} in Section \ref{embed}. In order to solve $K(t)$ and $K$, we give Algorithm \ref{algorithm1} and Algorithm \ref{algorithm2} in Section \ref{test} respectively.

\section{The (sufficient) existence of $K^*$}\label{Sufficient}
\ \ \ \ We want to know under what conditions the functional $J[K]$ reaches the minimum, that is, find the minimizer $K^*$ to make the functional reach the minimum.
\subsection{When $K^*\in\Omega$}\label{K in Omega}
\ \ \ \ In functional analysis, a general approach is to find a family of functions near the extreme point, and then take the convergent subsequence to converge to the extreme point. This requires some compactness of space.


However, for this infinite dimensional space $\Omega$, the boundedness is not necessarily completely-continuous, so compactness cannot be achieved. Generally, the compactness needs to use Arzel\`{a}-Ascoli lemma, 
but this is not available in our situation, because the family of functions cannot be equicontinuous. Therefore, we can only find the minimum of the functional in an approximate sense (approximate minimum). 
Combining Section \ref{The main idea} \textbf{(i)} and \textbf{(ii)}, we give the following theorem:
\begin{theorem}\label{Approximate minimum}
\textbf{(Approximate minimum)}
For all $\epsilon>0$, there exist $K^\epsilon\in\Omega$ and $K_0\in\mathcal{L}^2(\Omega)$ such that $K^\epsilon\rightharpoonup K_0$ \footnote{$``\rightharpoonup"$ means weak convergence, which can also be represented by symbols $``\overset{w}{\longrightarrow}"$.} and $J[K_0]=\underset{K\in\Omega}{\inf}~J[K]$.
\end{theorem}
\begin{proof}
Since functional $J[K]$ has a lower bound, it has an infimum. It means that there exists $\{K_m\}_{m=1}^{+\infty}\in\Omega$ s.t. $\underset{m\rightarrow\infty}{\lim}J[K_m]=\underset{K\in\Omega}{\inf}~J[K]$. $\{K_m\}_{m=1}^{+\infty}$ is called the minimizing sequence of functional $J[\cdot]$. In the light of the definition of limit, we obtain that $\forall ~\epsilon>0,~\exists ~N\in\N^+$, and when $m>N$,
\begin{equation}
  |J[K_m]-\underset{K\in\Omega}{\inf}~J[K]|\leq\epsilon\Rightarrow J[K_m]\leq\underset{K\in\Omega}{\inf}~J[K]+\epsilon.
\end{equation}
According to Ekeland variational principle, let $\lambda=\sqrt\epsilon$, we know that there exists $K^\epsilon\in \Omega$ s.t.
\begin{center}
$\left\{ \begin{aligned}
i)~&J[K^\epsilon]\leq J[K_m],\\
ii)~&\parallel K^\epsilon-K_m\parallel_{\Omega}\leq\sqrt\epsilon,\\
iii)~&J[K^\epsilon]\leq J[K]+\sqrt\epsilon\parallel K^\epsilon-K\parallel_{\Omega},\forall ~K\in\Omega.
\end{aligned}\right.$
\end{center}
When $\epsilon\rightarrow0$, 
$J[K^\epsilon]\leq J[K],~\forall ~K\in\Omega$. Therefore, $K^\epsilon$ satisfies the property of approximate minimum of functional $J[\cdot]$.

In combination with Section \ref{The main idea} \textbf{(i)}, $||K^\epsilon||_\Omega\leq M<+\infty$, $||K_m||_\Omega\leq M<+\infty$. In other words, they are contained in a closed ball of $\Omega$ respectively. Although $\Omega$ is not compact, it is weakly sequentially closed in the sense of $\L^2$. $\L^2(\Omega)$ is a Hilbert space.
So there exists a weakly convergent subsequence $\{K_{n_i}\}_{i=1}^{+\infty}$ s.t. $\underset{i\rightarrow\infty}{\lim}J[K_{n_i}]=\underset{K\in\Omega}{\inf}~J[K]$ and $K_{n_i}\overset{i\rightarrow\infty}{\rightharpoonup}$
$K_0,~K_0\in\L^2(\Omega)$. This means that $K^\epsilon\rightharpoonup K_0$ in $\L^2(\Omega)$. According to the lower-semi-continuity of functional and Fatou theorem 
, $J[K_0]\leq\underset{i\rightarrow\infty}{\underline{\lim}}~J[K_{n_i}]$, we have
\begin{equation}
  \underset{K\in\Omega}{\inf}~J[K]\leq J[K_0]\leq\underset{i\rightarrow\infty}{\underline{\lim}}~J[K_{n_i}]=\underset{i\rightarrow\infty}{\lim}~J[K_{n_i}]=\underset{K\in\Omega}{\inf}~J[K],
\end{equation}
i.e. 
$J[K_0]=\underset{K\in\Omega}{\inf}~J[K]$.

\end{proof}
It is significant to reiterate that the convergent subsequence can be extracted in $\Omega$, if $\Omega$ is a compact space (Section \ref{The main idea} \textbf{(iii)}).
This means that \textbf{the minimizer $K_0$} indeed exists, i.e. $K_0\in\Omega$.
For example, in polynomial space, 
the results corresponding to Theorem \ref{Approximate minimum} can be written as the following.
\begin{theorem}\textbf{(The minimizer in Polynomial space)}\label{Polynomial space's minimizer}
For a given positive integer $m$, let $P=\{K|K(x(t,x_0))=f_m(x(t,x_0))\}$ with the form
\begin{equation*}
f_m(x(\cdot))=\underset{\alpha_1+\alpha_2+\cdots+\alpha_n\leq m}{\sum}\left(\begin{aligned}
&a_{\alpha_1\alpha_2\cdots\alpha_n}^1\\
&a_{\alpha_1\alpha_2\cdots\alpha_n}^2\\
&\cdots\\
&a_{\alpha_1\alpha_2\cdots\alpha_n}^n
\end{aligned}\right)x_1^{\alpha_1}x_2^{\alpha_2}\cdots x_n^{\alpha_n}\triangleq \underset{\alpha_1+\alpha_2+\cdots+\alpha_n\leq m}{\sum}A^mx_1^{\alpha_1}x_2^{\alpha_2}\cdots x_n^{\alpha_n},
\end{equation*}
where $\alpha_1,\alpha_2,\cdots,\alpha_n\in\{0,1,\cdots,m\}$, $\{a_{\alpha_1\alpha_2\cdots\alpha_n}^j\}_{j=1}^n$ is coefficients of $x_1^{\alpha_1}x_2^{\alpha_2}\cdots x_n^{\alpha_n}$ for the $j$-th row, $x(\cdot)\in\R^n,~\{x_i\}_{i=1}^n$ is the $i$-th component of $x(\cdot)$. Then functionals (\ref{conjugation functional}) reaches minimum. Additionally, the similarity and semi-similarity between two systems can be realized in $\tilde{\Omega}=\Omega\cap P$. More precisely, $K_0\in\tilde{\Omega}$.
\end{theorem}
\begin{proof}
Compared with the proof of Theorem \ref{Approximate minimum}, we only need to further prove that convergent subsequences can be extracted in $\tilde{\Omega}$.

Define $F(f_m(\cdot))=A^m$. 
According to Theorem \ref{Approximate minimum}, there exists a weakly convergent subsequence $\{K_{n_i}\}\in\Omega$ s.t. $K_{n_i}\overset{i\rightarrow\infty}{\rightharpoonup}K_0$. Therefore, $F(K_{n_i})\overset{i\rightarrow\infty}{\rightarrow}F(K_0)$, i.e. $A_{n_i}^m\triangleq A_i^m\overset{i\rightarrow\infty}{\rightarrow}A_0$. Then we only need to extract subsequences from $A_i^m$.

For a fixed $i$, $\{A_i^m\}$ is finite, thus there must be a convergent subsequence. Aggregate all of them and record them as $\{A_{i(j)}\}_{j=1}^\infty$, we obtain the convergent subsequences $A_{i(j)}\overset{j\rightarrow\infty}{\rightarrow}A_0$, and this implies $K_0\in\tilde{\Omega}$.
\end{proof}
\subsection{When $K^*(t)\in GL(n)\cap\Omega$}\label{K(t) in GL(n)}
\ \ \ \ As mentioned in Section \ref{The main idea} \textbf{(iii)}, we consider the homeomorphism $K(t)\in GL(n)\cap\Omega$. The cost functional (\ref{conjugation functional}) changed into
\begin{equation}
J_1[K(\cdot)]\triangleq\dfrac{1}{T} \int_{0}^{T}||K(t)x(t,x_0 )-y(t,K(0)x_0)||^2dt.
\end{equation}
\begin{theorem}\label{exists in GL(n)}
There exists a homeomorphism $K_0(t)\in GL(n)\cap\Omega$ such that $J_1[K_0(\cdot)]=0$.
\end{theorem}
\begin{proof}
$GL (n)$ is a complete metric space which satisfies the condition in Theorem \ref{Approximate minimum}. $K(t)$ contains $n^2$ elements $k_{ij}(t)~(i, j = 1,\cdots, n):\R\rightarrow\R$, $GL (n)$ has $n^2$ bases. In fact, any finite dimensional space is isomorphic. According to Theorem \ref{Polynomial space's minimizer}, there exists extremum point $K_0(t),~K_0(t)\in GL(n)$, and $K_0(t)$ is bounded.

To prove that $J_1[K_0(\cdot)]=0$, geometrically, it suffices to prove that $K_0$ maps the solution curve of $(x(t),x_0)$ to the solution curve of $(y(t),y_0)$. In order to find the $K_0$, we divide into three steps: (i) The Euler polyline is used to approximate the solution curve; (ii) $K_0(t)$ is solved by making the functional be $0$ on the Euler polyline; (iii) Certificate the uniform convergence of the Euler polyline.\\

\textbf{Step 1.} Firstly, we introduce the time variable to ensure that the direction vector at each point of the solution curve is not equal to $\mathbf{0}$:
$$
	\left\{ \begin{aligned}
		\overset{.}x(t) &=f(t,x(t)),\\
		          \dot{t}&=1, \\ \end{aligned}  \right.~
	\left\{  \begin{aligned}
		\overset{.}y(t) &=g(t,y(t)),\\
	             \dot{t}&=1. \\  \end{aligned}  \right.
$$
Then $X=(x,t)\in\R^{n+1},Y=(y,t)\in\R^{n+1}$, and $F=(f,1),G=(g,1)$. For simplicity, we will still consider $(x, y, f, g)$ instead of $(X, Y, F, G)$.

Divide $[0, T]$ into $m$ segments, each segment is $h_m=\frac{T}{m}$ long. Let $\tau_k=kh_m,~k=0,1,\cdots,m$,
\begin{equation}\begin{aligned}
&\xi_0=x_0,\quad\frac{\xi_{k+1}-\xi_k}{\tau_{k+1}-\tau_k}=f(\tau_k,\xi_k),\quad i.e.\quad\xi_{k+1}=\xi_k+h_mf(\tau_k,\xi_k),\\
&\eta_0=y_0,\quad\frac{\eta_{k+1}-\eta_k}{\tau_{k+1}-\tau_k}=g(\tau_k,\eta_k),\quad i.e.\quad\eta_{k+1}=\eta_k+h_mg(\tau_k,\eta_k).
\end{aligned}
\end{equation}
Thereupon, we get $2(m+1)$ discrete points: $$(\tau_0,\xi_0),(\tau_1,\xi_1),\cdots,(\tau_m,\xi_m);~(\tau_0,\eta_0),(\tau_1,\eta_1),\cdots,(\tau_m,\eta_m).$$
Connect two adjacent points with straight lines in turn to obtain two continuous Euler polylines:
\begin{equation}\begin{aligned}
&\varphi_m(t)=\left\{\begin{aligned}
x_0+f(\tau_0,\xi_0)t,\qquad&t\in[0,\tau_1],\\
x_0+h_m\underset{k=0}{\overset{l-1}{\Sigma}}f(\tau_k,\xi_k)+f(\tau_l,\xi_l)(t-\tau_l),~&t\in(\tau_l,\tau_{l+1}],l=1,\cdots,m-1,
\end{aligned}\right.\\
&\psi_m(t)=\left\{\begin{aligned}
y_0+g(\tau_0,\eta_0)t,\qquad&t\in[0,\tau_1],\\
y_0+h_m\underset{k=0}{\overset{l-1}{\Sigma}}g(\tau_k,\eta_k)+g(\tau_l,\eta_l)(t-\tau_l),~&t\in(\tau_l,\tau_{l+1}],l=1,\cdots,m-1.
\end{aligned}\right.
\end{aligned}
\end{equation}

\textbf{Step 2.} Seek $K_0(t)$ s.t. $K_0(t)\varphi_m(t)=\psi_m(t)$. Without losing generality, we translate the above $2(m+1)$ discrete points and their own direction vectors to the origin $O$, and then consider the rotation matrix of the direction vectors.

When $t\in[0,\tau_1]$, $K_0(t)=K_0$,
\begin{equation}\begin{aligned}
&\varphi_m(t)=x_0+f(\tau_0,\xi_0)t\triangleq x_0+\textbf{f}_0t,\\
&\psi_m(t)=y_0+g(\tau_0,\eta_0)t\triangleq y_0+\textbf{g}_0t,
\end{aligned}
\end{equation}
where $\textbf{f}_0\neq\mathbf{0},~\textbf{g}_0\neq\mathbf{0}$ are direction vectors of $\varphi_m(t),~\psi_m(t)$ respectively,
\begin{equation*}\begin{aligned}
&\textbf{f}_0=||\textbf{f}_0||(\cos{\alpha_0^1},\cos{\alpha_0^2},\cdots,\cos{\alpha_0^n})^T\triangleq||\textbf{f}_0||(u_0^1,u_0^2,\cdots,u_0^n)^T\triangleq||\textbf{f}_0||\textbf{u}_0,\\ &\textbf{g}_0=||\textbf{g}_0||(\cos{\beta_0^1},\cos{\beta_0^2},\cdots,\cos{\beta_0^n})^T\triangleq||\textbf{g}_0||(v_0^1,v_0^2,\cdots,v_0^n)^T\triangleq||\textbf{g}_0||\textbf{v}_0, \end{aligned}
\end{equation*}
where $\alpha_0^i,~\beta_0^i,~i=1,2,\cdots,n$, represent the angles between $\textbf{f}_0,~\textbf{g}_0$ and the coordinate axis. Leveraging the fact that $\underset{i=1}{\overset{n}{\Sigma}}\cos^2{\alpha_0^i}=1,~\underset{i=1}{\overset{n}{\Sigma}}\cos^2{\beta_0^i}=1$, $\textbf{u}_0,~\textbf{v}_0$ are unit vectors. Now let us start to calculate the rotation matrix $P_0$ of unit vector $\textbf{u}_0$ to unit vector $\textbf{v}_0$.

A rotation can be regarded as rotation around an axis on a plane. Figure \ref{rotation} shows the vector rotation process when $n=3$. Without losing generality, we first assume that $\textbf{u}_0$ rotates on the plane formed by $u_0^1\times u_0^2$, the rotation angle is $\theta_1,~\cos{\theta_1}=\frac{u_0^2}{\sqrt{(u_0^1)^2+(u_0^2)^2}},~\sin{\theta_1}=\frac{u_0^1}{\sqrt{(u_0^1)^2+(u_0^2)^2}}$, which can be expressed as:
\begin{figure}[htbp]
\centering
	\includegraphics[width=0.75\linewidth]{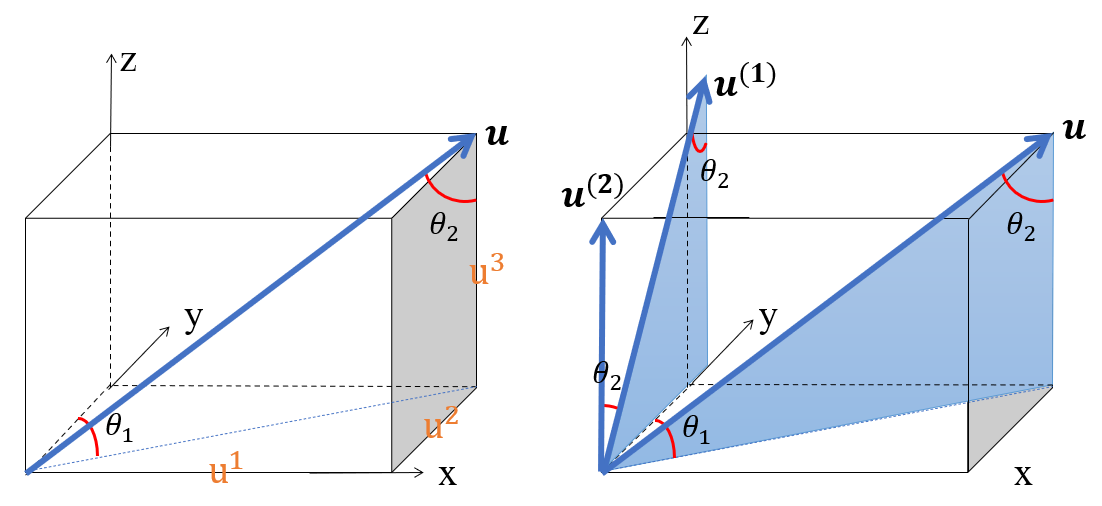}
	\caption{\small{Vector rotation process when $n=3$.}}
{\small{
As illustrated in the figure, the blue plane rotates $\theta_1$ to get the vector $u^{(1)}$. Then $u^{(1)}$ rotates $\theta_2$ to get the vector $u^{(2)}$.}}
	\label{rotation}
\end{figure}
\small{\begin{equation*}
\textbf{u}_0^{(1)}=R_1\textbf{u}_0=\(
\begin{matrix}
\cos{\theta_1}&-\sin{\theta_1}&&&\\
\sin{\theta_1}&\cos{\theta_1}&&&\\
&&1&&\\
&&&\ddots&\\
&&&&1
\end{matrix}\)
\(\begin{matrix}
u_0^1\\
u_0^2\\
u_0^3\\
\vdots\\
u_0^n
\end{matrix}\)=
\(\begin{matrix}
0\\
\sqrt{(u_0^1)^2+(u_0^2)^2}\\
u_0^3\\
\vdots\\
u_0^n
\end{matrix}\).
\end{equation*}}
Deduce the rest, $\textbf{u}_0^{(1)}$ rotates on the plane formed by $u_0^{2'}\times u_0^{3'}$, the rotation angle is $\theta_2,~\cos{\theta_2}=\frac{u_0^3}{\sqrt{(u_0^1)^2+(u_0^2)^2+(u_0^3)^2}},~\sin{\theta_2}=\frac{\sqrt{(u_0^1)^2+(u_0^2)^2}}{\sqrt{(u_0^1)^2+(u_0^2)^2+(u_0^3)^2}}$,
\small{\begin{equation*}\begin{aligned}
\textbf{u}_0^{(2)}&=R_2\textbf{u}_0^{(1)}=\(
\begin{matrix}
1&&&&&\\
&\cos{\theta_2}&-\sin{\theta_2}&&&\\
&\sin{\theta_2}&\cos{\theta_2}&&&\\
&&&1&&\\
&&&&\ddots&\\
&&&&&1
\end{matrix}\)
\(\begin{matrix}
0\\
\sqrt{(u_0^1)^2+(u_0^2)^2}\\
u_0^3\\
u_0^4\\
\vdots\\
u_0^n
\end{matrix}\)=
\(\begin{matrix}
0\\
0\\
\sqrt{(u_0^1)^2+(u_0^2)^2+(u_0^3)^2}\\
u_0^4\\
\vdots\\
u_0^n
\end{matrix}\),\\
&\cdots\cdots
\end{aligned}
\end{equation*}
\begin{equation*}\begin{aligned}
\textbf{u}_0^{(n-1)}=R_{n-1}\textbf{u}_0^{(n-2)}&=\(
\begin{matrix}
1&&&&\\
&\ddots&&&\\
&&1&&\\
&&&\cos{\theta_{n-1}}&-\sin{\theta_{n-1}}\\
&&&\sin{\theta_{n-1}}&\cos{\theta_{n-1}}
\end{matrix}\)
\(\begin{matrix}
0\\
\vdots\\
0\\
\sqrt{(u_0^1)^2+(u_0^2)^2+\cdots+(u_0^{n-1})^2}\\
u_0^n
\end{matrix}\)\\
&=\(\begin{matrix}
0\\
\vdots\\
0\\
\sqrt{(u_0^1)^2+(u_0^2)^2+\cdots+(u_0^n)^2}
\end{matrix}\)=\(\begin{matrix}
0\\
\vdots\\
0\\
1
\end{matrix}\).
\end{aligned}
\end{equation*}}
Let $R_0=R_{n-1}\cdots R_2R_1$, $R_0\textbf{u}_0=(0,\cdots,0,1)^T$. Then $R_0$ is invertible, and its inverse is $R_0^{-1}=R_0^T=R_1^TR_2^T...R_{n-1}^T$. Note that $R_i,~i=1,\cdots ,n-1$, are ordered and not exchangeable, and
each $R_i$ is an orthogonal matrix. Similarly, $\exists ~Q_0=Q_{n-1}\cdots Q_2Q_1$, $Q_0\textbf{v}_0=(0,\cdots ,0,1)^T$. $Q_0$ is invertible, and its inverse is $Q_0^{-1}=Q_0^T=Q_1^TQ_2^T\cdots Q_{n-1}^T$. $R_0\textbf{u}_0=Q_0\textbf{v}_0=(0,\cdots ,0,1)^T$, let $P_0=Q_0^{-1}R_0$, then $P_0\textbf{u}_0=\textbf{v}_0$, $P_0$ is invertible. Let $K_0=\frac{||\textbf{g}_0||}{||\textbf{f}_0||}P_0t$, consequently,
\begin{equation}
K_0(x_0+\textbf{f}_0t)=y_0+\textbf{g}_0t,K_0x_0=y_0~(\text{translation}),~K_0f_0=g_0~(\text{rotation}).
\end{equation}

When $t\in(\tau_l,\tau_{l+1}], ~K_0(t)=K_l,~l=1,\cdots ,m-1$,
\begin{equation}\begin{aligned}
&\varphi_m(t)=x_0+h_m\underset{k=0}{\overset{l-1}{\Sigma}}f(\tau_k,\xi_k)+f(\tau_l,\xi_l)(t-\tau_l)\triangleq \textbf{f}_lt+x_0+h_m\underset{k=0}{\overset{l-1}{\Sigma}}\textbf{f}_k-\textbf{f}_l\tau_l,\\
&\psi_m(t)=y_0+h_m\underset{k=0}{\overset{l-1}{\Sigma}}g(\tau_k,\eta_k)+g(\tau_l,\eta_l)(t-\tau_l)\triangleq \textbf{g}_lt+y_0+h_m\underset{k=0}{\overset{l-1}{\Sigma}}\textbf{g}_k-\textbf{g}_l\tau_l,
\end{aligned}
\end{equation}
where $\textbf{f}_l\neq\mathbf{0},~\textbf{g}_l\neq\mathbf{0}$ are direction vectors of $\varphi_m(t),~\psi_m(t)$ respectively,
\begin{equation*}\begin{aligned}
\textbf{f}_l&=||\textbf{f}_l||(\cos{\alpha_l^1},\cos{\alpha_l^2},\cdots ,\cos{\alpha_l^n})^T\triangleq||\textbf{f}_l||(u_l^1,u_l^2,\cdots ,u_l^n)^T\triangleq||\textbf{f}_l||\textbf{u}_l, \\ \textbf{g}_l&=||\textbf{g}_l||(\cos{\beta_l^1},\cos{\beta_l^2},\cdots ,\cos{\beta_l^n})^T\triangleq||\textbf{g}_l||(v_0^1,v_l^2,\cdots ,v_l^n)^T\triangleq||\textbf{g}_l||\textbf{v}_l,
\end{aligned}
\end{equation*}
where $\alpha_l^i,~\beta_l^i,~i=1,2,\cdots ,n$, represent the angles between $\textbf{f}_l,~\textbf{g}_l$ and the coordinate axis. Leveraging the fact that $\underset{i=1}{\overset{n}{\Sigma}}\cos^2{\alpha_l^i}=1,~\underset{i=1}{\overset{n}{\Sigma}}\cos^2{\beta_l^i}=1$, $\textbf{u}_l,~\textbf{v}_l$ are unit vectors.

Use the same rotation transformation, there exist
orthogonal rotation matrices $R_l,~Q_l$ s.t. $R_l\textbf{u}_l=Q_l\textbf{v}_l=(0,\cdots ,0,1)^T$. Let $P_l=Q_l^{-1}R_l$, then $P_l\textbf{u}_l=\textbf{v}_l$, $P_l$ is invertible. Let $K_l=\frac{||\textbf{g}_l||}{||\textbf{f}_l||}P_lt$, consequently,
\begin{equation}\left\{\begin{aligned}
&K_l(x_0+h_m\underset{k=0}{\overset{l-1}{\Sigma}}\textbf{f}_k+\textbf{f}_l(t-\tau_l))=y_0+h_m\underset{k=0}{\overset{l-1}{\Sigma}}\textbf{g}_k+\textbf{g}_l(t-\tau_l),\\
&K_l(x_0+h_m\underset{k=0}{\overset{l-1}{\Sigma}}\textbf{f}_k-\textbf{f}_l\tau_l)=y_0+h_m\underset{k=0}{\overset{l-1}{\Sigma}}\textbf{g}_k-\textbf{g}_l\tau_l~(\text{translation}),\\
&K_l\mathbf{f_l}=\mathbf{g_l}~(\text{rotation}).
\end{aligned}\right.\end{equation}
Thereupon, $K_0(t)=K_l$, when $t\in[\tau_l,\tau_{l+1}],~l=0,\cdots ,m-1$. It is noteworthy that 
$K_l,l=0,1,\cdots,m-1$ is bounded, reversible, and its inverse $K_l^{-1}$ is bounded. So $K_0(t)$ is the homeomorphism (i.e. $K_0(t)\in GL(n)\cap\Omega$). \\

\textbf{Step 3.} Define $\triangle_m^1(t)=\varphi_m(t)-(x_0+\int_{0}^{T}f(s,\varphi_m(s))ds),~\triangle_m^2(t)=\psi_m(t)-(y_0+\int_{0}^{T}g(s,\psi_m(s))ds)$.
It bears out that these two differences converge uniformly by Arzel$\grave{a}$-Ascoli lemma, see \cite{[WL+12]}: $\triangle_m^1(t)\underset{1}{\overset{m\rightarrow\infty}{\longrightarrow}}0,~\triangle_m^2(t)\underset{1}{\overset{m\rightarrow\infty}{\longrightarrow}}0$. \footnote{$f_m(x)\underset{1}{\overset{m\rightarrow\infty}{\longrightarrow}}f_0(x)$ means that $f_m(x)$ uniformly (stronger than pointwise) converges to $f_0(x)$, which is independent of $x$.}
This means that there exist convergent subsequences which are still recorded as $\varphi_m(t),\psi_m(t)$, s.t. $\varphi_m(t)\underset{1}{\overset{m\rightarrow\infty}{\longrightarrow}}x(t),~\psi_m(t)\underset{1}{\overset{m\rightarrow\infty}{\longrightarrow}}y(t)$, thus $K_0(t)x(t)=y(t)$. $K_0(t)$ depends on continuous  $\varphi_m(t),~\psi_m(t)$, so $K_0(t)$ is continuous. Therefore,
\begin{equation}\begin{aligned}
J_1[K_0(\cdot)]&=\dfrac{1}{T} \int_{0}^{T}||K_0(t)x(t,x_0 )-y(t,K_0(0)x_0)||^2dt\\
&=\dfrac{1}{T} \int_{0}^{T}||K_0(t)x(t,x_0 )-K_0(t)\varphi_m(t)+K_0(t)\varphi_m(t)-\psi_m(t)+\psi_m(t)-y(t,K_0(0)x_0)||^2dt\\
&\leq C_1\int_{0}^{T}||x(t,x_0)-\varphi_m(t)||^2dt+C_2\int_{0}^{T}||K_0(t)\varphi_m(t)-\psi_m(t)||^2dt+C_3\int_{0}^{T}||\psi_m(t)-y(t,y_0)||^2dt\\
&=I_1+C_2\sum_{k=0}^{m-1}\int_{\tau_k}^{\tau_{k+1}}||K_0(t)\varphi_m(t)-\psi_m(t)||^2dt+I_3~\overset{m\rightarrow\infty}{\longrightarrow}0.
\end{aligned}
\end{equation}
\end{proof}
Let $D_1,D_2$ be two disjoint bounded regions in $\R^n$, $x_0\in D_1,y_0\in D_2$.
\begin{proposition}\label{K*(t,x)}
There exists a homeomorphism $K^*(t,x_0)\in \Omega$ such that $J[K^*(\cdot)]=0$.
\end{proposition}
\begin{proof}
On the basis of Theorem \ref{exists in GL(n)}, the proof is very direct.
Firstly, there exists a homeomorphism $H:D_1\rightarrow D_2$: $\forall~y_0\in D_2,~\exists$ a unique $x_0$ s.t. $H(x_0)=y_0,~x_0=H^{-1}(y_0)$ (Poincar\'{e} Theorem \cite{[Poincare]}).

Let $K^*(t,x_0)=K_l\circ H$, when $t\in[\tau_l,\tau_{l+1}],~l=0,\cdots ,m-1$. According to the continuous dependence of the solution on the initial value, $K_l$ is continuous for $x_0$, so $K^*(t,x_0)$ is also continuous for $x_0$. Then $K^*(t,x_0)$ is also a homeomorphism (i.e. $K^*(t,x_0)\in\Omega$), and $J[K^*(\cdot)]=0$.
\end{proof}

Further, we will consider $x(t_1)=x_1,y(t_1)=y_1$ on the trajectories $x(t),y(t)$ when the terminal $t=t_1\in[0,T]$ is given. We will discuss the form of $K^*$ when $K^*x_1=y_1$ in the next section.
\subsection{When the terminal is given}\label{When K in H}
\ \ \ \ Review the conjugate system $K^*(e^{At}x_0)=y(t,K^*(x_0))$. For (\ref{Hartman y}) and (\ref{Hartman x}), Proposition \ref{Hartman} gives the form of $K^*$ when the nonlinear term $r$ has a small Lipschitz constant $|r|_{Lip}$ in (\ref{r Lipschitz}).
The homeomorphism $K^*=I_X+g\in\Omega$ is not explicitly represented by $t$, but is implicitly represented by $x (t)$, where $g(x)=\int_\R G_A(s)r(e^{-As}x+g(e^{-As}x))ds$.
For (\ref{2nonlinear}), Theorem \ref{exists in GL(n)} explicitly expresses $K^*$ as $K^*(t)$ and proves the form of $K^*(t)\in GL(n)\cap\Omega$.

When the terminal $t=t_1\in[0,T]$ is given, we study $x(t_1)=x_1,y(t_1)=y_1$ on the trajectories $x(t),y(t)$, the value of $K^*(x_1)$ will change as $t_1$ changes. For general nonlinear systems (\ref{2nonlinear}), Theorem \ref{exists in GL(n)} gives a way for $K^*(t)$ to translate and rotate over $t$, and makes $K^*(t_1)x_1=y_1$, which means that $y_1$ can indeed be obtained.

However, for (\ref{Hartman y}) and (\ref{Hartman x}) in the Hartman-Grobman theorem, which can be regard as a special case of (\ref{2nonlinear}), there is some difference. The $K^*$ should be independent of $t\in[0,t_1]$, and it should only depend on the terminal $t_1$. 

In this section, 
we will prove that the form of $K^*$ is actually expressed as: \begin{equation}\label{Hartman K(t,x)}
K^*(\cdot)=U^T(t_1,x_0)G^{-1}(t_1,x_0)(e^{-At_1}y_1-y_0),
\end{equation}
where
$$U(s,x_0)=e^{As}\triangledown r(\mathbf{0})x_0,$$
$$G(t,x_0)=\int_0^tU(s,x_0)U^T(s,x_0)ds.$$\\

Note that $y(t)=e^{At}y_0+\int_0^te^{A(t-s)}r(y(s))ds$ is the solution of (\ref{Hartman y}), $x(t)=e^{At}x_0$ is the solution of (\ref{Hartman x}). Nearby the hyperbolic equilibrium ($x^*=y^*=\mathbf{0}$), $r(y)=r(\mathbf{0})+\triangledown r(\mathbf{0})y+o(|y|)$\footnote{The symbol $\triangledown$ represents the gradient, and $\triangledown r(y)=\Big(\frac{\partial r(y)}{\partial y_1},\frac{\partial r(y)}{\partial y_2},\cdots,\frac{\partial r(y)}{\partial y_n}\Big)$ is a matrix, where $r(y)=(r(y)_1,r(y)_2,\cdots,r(y)_n)$.}. Ignore higher order terms and $r(\mathbf{0})=\mathbf{0}$, then $r(y(s))=\triangledown r(\mathbf{0})y(s)=\triangledown r(\mathbf{0})K^*(s,x_0)x(s)$ and
\begin{equation}\label{Hartman y(t)}
y(t)=e^{At}y_0+\int_0^te^{At}\triangledown r(\mathbf{0})K^*(s,x_0)x_0ds.
\end{equation}

We observed that when $x_0=(0,\cdots,0,1)$, $\det{G(t,x_0)}=0$, $G(t,x_0)$ is singular, and when $\triangledown r(\mathbf{0}) x_0=(1,\cdots,1)$, $\det{G(t,x_0)}\neq0$, $G(t,x_0)$ is non-singular (invertible). However,
in the Step 2 of Theorem \ref{exists in GL(n)}, we find the invertible $K^*$ through appropriate invertible translation and rotation. Hence, in order to obtain the invertible $K^*$, we suppose that $G(t,x_0)$ is invertible here.

According to the proof in Proposition \ref{K*(t,x)}, there exists a homeomorphism $H$ s.t. $H(x_0)=y_0$ and $H(x(t_1))\triangleq H(x_1)=y_1\triangleq y(t_1),~\forall~t_1\in[0,T]$. On the basis of the theory of controllability \cite{[WU+17]}, if $G(t,x_0)$ is invertible, the exhibit form of $K^*$ is
\begin{equation}\label{K(t,x)}
K^*(t,x_0)=U^T(t,x_0)G^{-1}(t_1,x_0)(e^{-At_1}y_1-y_0).
\end{equation}

Next, we prove the existence of $K^*$ in (\ref{K(t,x)}) by the Schauder fixed point theorem \cite{[DK+70]}. Substituting $K^*(t,x_0)$ into the right-hand side of (\ref{Hartman y(t)}) yields the nonlinear operator
\begin{equation}\label{Def P}\begin{aligned}
y(t)&=e^{At}y_0+\int_0^te^{At}\triangledown r(0)U^T(s,x_0)G^{-1}(t_1,x_0)(e^{-At_1}y_1-y_0)x_0ds\\
&=e^{At}y_0+e^{At}G(t,x_0)G^{-1}(t_1,x_0)(e^{-At_1}y_1-y_0)\triangleq Py(t).
\end{aligned}
\end{equation}
Clearly $Py(0)=y_0$ and $Py(t_1) = y_1$. Thus if a fixed point of the operator $P$ is
obtained the $K^*(t,x_0)$ (\ref{K(t,x)}) steers the system (\ref{Hartman x}) from $x_0$ to $x_1$, in time $t_1$. Then we have the following theorem.
\begin{theorem}\label{terminal K}
Suppose that $x(t)$, $y(t)$ satisfy equation (\ref{Hartman x}) and equation (\ref{Hartman y}), and $A\in \R^{n\times n}$ is a constant matrix s.t. $|e^{At}|\leq Me^{-\eta|t|},t\in\R$ with some constants $M\geq1$ and $\eta>0$.
Let $r(y)=o(|y|)\in C^1(\R^n)$ such that $|\triangledown r(y)|\leq c_1$, where $c_1\geq0$.
Further, suppose that $G(t,x_0)$ is an invertible matrix and there exists a constant $c_2>0$ s.t. $\underset{x\in C[0,t_1;\R^n]}{\inf}|\det G(t,x)|\geq \frac{1}{c_2}$.
Then for every pair $x_0,x_1\in\R^n$, the operator $P$ defined by (\ref{Def P}) has a fixed point in $C[0,t_1;\R^n]$.
\end{theorem}
\begin{proof}
Define the norm in $C[0,T;\R^n]$ be $||y(t)||=\underset{t\in[0,T]}{\max}|y(t)|$. According to (\ref{Def P}),
\begin{equation}\label{Py(t)}\begin{aligned}
|Py(t)|&\leq Me^{-\eta t}|y_0|+Me^{-\eta t}\int_0^t\mid Me^{-\eta s}c_1|x_0|\mid^2ds
\Big(c_2(Me^{-\eta t_1}|y_1|+|y_0|)\Big)\\
&= Me^{-\eta t}|y_0|+Me^{-\eta t}(Mc_1|x_0|)^2\frac{1}{2\eta}(1-e^{-2\eta t})
\Big(c_2(Me^{-\eta t_1}|y_1|+|y_0|)\Big)\\
&\leq Me^{-\eta t}|y_0|+\frac{M^3 c_1^2 c_2|x_0|^2}{2\eta}e^{-\eta t}
\Big(Me^{-\eta t_1}|y_1|+|y_0|\Big).
\end{aligned}\end{equation}
Due to $e^{-\eta t}\leq1$,
\begin{equation}\label{C}
|Py(t)|\leq M|y_0|+\frac{M^3 c_1^2 c_2|x_0|^2}{2\eta}\Big(Me^{-\eta t_1}|y_1|+|y_0|\Big)\triangleq C_0.
\end{equation}

Consider the closed and convex subset $\chi\triangleq\{y(t)|y(t)\in C[0,t_1;\R^n],||y(t)||=\underset{t\in[0,t_1]}{\max}|y(t)|\leq C_0\}$ of $C[0,t_1;\R^n]$. Let $\Xi\triangleq\{z(t)|z(t)=P(y(t)),y(t)\in \chi\}$ be the image of $\chi$.
$\forall~\epsilon>0,\forall~t',t''\in[0,t_1]$ s.t. $|t''-t'|<\delta$,
\begin{equation}\label{euqual-con}\begin{aligned}
|z(t'')-z(t')|\leq |e^{At''}&-e^{At'}|~|y_0|+|e^{At''}G(t'',x_0)-e^{At'}G(t',x_0)|~|G^{-1}(t_1,x_0)(e^{-At_1}y_1-y_0)|\\
\leq |Ae^{At_\xi}&|~|t''-t'|~|y_0|\\
+&\bigg|\int_{t'}^{t''}[e^{At}U(t)U^T(t)+Ae^{At}\int_0^tU(s)U^T(s)ds]dt\bigg|
|G^{-1}(t_1,x_0)(e^{-At_1}y_1-y_0)|\\
\leq |AM&e^{-\eta t_\xi}|~|t''-t'|~|y_0|\\
+&\bigg|\int_{t'}^{t''}[M^3c_1^2|x_0|^2e^{-3\eta t}+A\frac{M^3 c_1^2|x_0|^2}{2\eta}e^{-\eta t}]dt\bigg|
\bigg(c_2(Me^{-\eta t_1}|y_1|+|y_0|)\bigg)\\
\leq M|A&y_0|~|t''-t'|\\
+&\bigg|M^3c_1^2|x_0|^2+A\frac{M^3 c_1^2|x_0|^2}{2\eta}\bigg|
\bigg(c_2(Me^{-\eta t_1}|y_1|+|y_0|)\bigg)|t''-t'|\\
\triangleq C_1|t''&-t'|\\
< C_1\delta&=\epsilon,\\
\end{aligned}
\end{equation}
where $\delta=\frac{\epsilon}{C_1},t_\xi\in[t',t'']$ (without losing generality, $t'<t''$). This shows that $z(t)$ is equicontinuous.\\

The operator $P$ as defined by (\ref{Def P}) is continuous and it is easily established from the Arzel\`{a}-Ascoli lemma that the image set $\Xi$ is compact and is a subset of $\chi$. In fact, it is only necessary to prove that there is a subsequence $\{z_{n_j}(t)\}_{j=1}^\infty$ of $\{z_n(t)\}_{n=1}^\infty\in\Xi$, which is uniformly convergent on $[0,t_1]$.

Take the rational sequence $\{r_i\}_{i=1}^\infty,r_i\in[0, t_1]$, and consider the infinite square matrix $z_n(r_i),n,i=1,2,\cdots$. They are uniformly bounded, i.e. $||z_n(r_i)||\leq C_0$ (\ref{C}). By using the diagonal method, there is a subsequence that converge at all $r_i$, which is recorded as $\{z_{n_j}(t)\}_{j=1}^\infty$.

On the one hand, the equicontinuity (\ref{euqual-con}) shows that $\forall~\epsilon>0$, $\exists~\delta>0$ s.t. when $|t''-t'|<\delta$,
$$
|z_{n_j}(t'')-z_{n_j}(t')|<\frac{\epsilon}{3},j=1,2,\cdots.
$$
On the other hand, we can find finite $r_i,i=1,2,\cdots,I$ s.t. $[0,t_1]\subset\underset{i=1}{\overset{I}{\bigcup}}(r_i-\delta,r_i+\delta)$. For each $r_i,i=1,2,\cdots,I$, $\exists~N=N(\epsilon)$ s.t. when $j,k\geq N$,
$$
|z_{n_j}(r_i)-z_{n_k}(r_i)|<\frac{\epsilon}{3}.
$$

Combined with the above analysis, $\forall~t\in[0,t_1],\exists~r_i(i=1,2,\cdots,I)$ s.t. $|t-r_i|<\delta$, when $j,k\geq N$,
$$\begin{aligned}
|z_{n_j}(t)-z_{n_k}(t)|&\leq|z_{n_j}(t)-z_{n_j}(r_i)|+|z_{n_j}(r_i)-z_{n_k}(r_i)|+|z_{n_k}(r_i)-z_{n_k}(t)|\\
&<\frac{\epsilon}{3}+\frac{\epsilon}{3}+\frac{\epsilon}{3}=\epsilon.
\end{aligned}
$$
This shows that $\{z_{n_j}(t)\}_{j=1}^\infty$ is uniformly convergent on $[0,t_1]$.
By the Arzel\`{a}-Ascoli lemma,
the image set $\Xi$ is compact and is a subset of the closed and convex set $\chi$.
Hence by the Schauder fixed point theorem, the operator $P$ has a fixed point in $C[0,t_1;\R^n]$.
\end{proof}

In order to obtain (\ref{Hartman K(t,x)}) and the properties of the hyperbolic fixed point $\mathbf{0}$, we replace $t$ with $t_1$ in (\ref{K(t,x)}) and (\ref{Py(t)}):
\begin{equation*}
K^*(\cdot)=U^T(t_1,x_0)G^{-1}(t_1,x_0)(e^{-At_1}y_1-y_0),
\end{equation*}
$$
|y_1|\leq Me^{-\eta t_1}|y_0|+\frac{M^3 c_1^2 c_2|x_0|^2}{2\eta}e^{-\eta t_1}
\Big(Me^{-\eta t_1}|y_1|+|y_0|\Big).
$$
Then
$$
\Big(1-\frac{M^4 c_1^2 c_2|x_0|^2}{2\eta}e^{-2\eta t_1}\Big)|y_1|\leq \Big(Me^{-\eta t_1}+\frac{M^3 c_1^2 c_2|x_0|^2}{2\eta}e^{-\eta t_1}\Big)|y_0|.
$$
We can choose $M,\eta$ s.t. $\frac{M^4 c_1^2 c_2|x_0|^2}{2\eta}e^{-2\eta t_1}<1$ and $|y_1|\leq q(t_1)|y_0|,$ where
\begin{equation}\label{q(t)}
q(t_1)=\frac{Me^{-\eta t_1}+\frac{M^3 c_1^2 c_2|x_0|^2}{2\eta}e^{-\eta t_1}}{1-\frac{M^4 c_1^2 c_2|x_0|^2}{2\eta}e^{-2\eta t_1}},0<q(t_1)<1.
\end{equation}
It can be verified that $q(t_1)\overset{t_1\rightarrow\infty}{\longrightarrow}0$, and this means that $y(t_1)\overset{t_1\rightarrow\infty}{\longrightarrow}\mathbf{0}=y^*=x^*$ (the hyperbolic equilibrium).

In general, the choice of a fixed-point theorem (see \cite{[CQ+88]}) will depend on the conditions imposed on $f(\cdot)$ and $g(\cdot)$ in (\ref{2nonlinear}).\\

So far, we have proved the existence of functional extremum. Back to our original problem, we hope to achieve the similarity between the two systems by the minimizer. Now the minimizer already exists, we focus on exploring the maximum principle, the necessary conditions for the minimizer.

\section{Necessary conditions-the maximum principle}\label{MAX}
\ \ \ \ In the usual manner, to determine the minimizer $K^*$, we need to derive a maximum principle as the necessary condition like Pontryagin maximum principle. The main idea is that the local variation of the functional $J[K(\cdot)]$ must be 0 at $K^*$ (i.e. $\delta J[K^*(\cdot)]=0$), if $J[K^*(\cdot)]$ is the extreme value, which is mentioned in Section \ref{The main idea} \textbf{(iv)}.

In this section, we assume that the functions $f(\cdot)$ and $g(\cdot)$ in (\ref{2nonlinear}) are twice continuous differentiable with respect to their state variables.
For a fixed initial value $x_0$, let $\Gamma_{x_0}$ be the trajectory of $x$.
We consider $K(t)\in\Omega$, $\Omega$ 
is a convex set, hence we use the Lagrange multiplier method to obtain the maximum principle.


\subsection{Maximum principle of two systems}\label{MAX1}
\ \ \ \ Introduce Lagrange multiplier vector valued functions $\lambda (t)$, $\mu
(t)$, and solve the functional minimum problem under the constraint of (\ref{2nonlinear}). Define a new functional $J_2$,
\begin{equation}
	J_2[K(\cdot)]=\dfrac{1}{T} \int_{0}^{T}\{||K(x(t,x_0)-y(t,y_0)||^2+\lambda^T
(t)[\dot{x}-f(t,x)]+\mu^T (t)[\dot{y}-g(t,y)] \}dt ,
\end{equation}
where $K(\cdot)\in \Omega$, $t\in [0,T]$, $y_0$ is given such that $y_0=K(x_0)$.
The minimizer is $$K^*(\cdot)=\arg\min~\dfrac{1}{T} \int_{0}^{T} ||K(x(t,x_0))-y(t,y_0)||^2 dt.$$
Write $H(t,x(t),y(t),K(\cdot),\lambda (t),\mu (t))$

$=-\dfrac{1}{T}||K(x(t,x_0) )-y(t,y_0)||^2+\lambda^T (t)f(t,x)+\mu^T
(t)g(t,y)  . $\\

\begin{theorem}\label{max1}
\textbf{ (Maximum principle)}
If $K^*(\cdot)\in \Omega$ such that $J_2 [K^*(\cdot)]$ is the minimum value, then there are
vector valued functions $\lambda (t),~\mu (t)\in \mathbb{R}^n $ satisfying
\begin{equation}\label{a1b1c1d1}
	\begin{aligned}
		(a)&\left\{ \begin{aligned}
			\overset{.}x(t) &=H_\lambda^T,  \\
			\overset{.}y(t) &=H_\mu^T, \\ \end{aligned}  \right. \\
		(b)&\left\{  \begin{aligned}
			\overset{.}\lambda^{T}(t) &=-\frac{\partial
				H(t,x,y,K^*(\cdot),\lambda(t),\mu(t))}{\partial x},\\
			\lambda^T(T)&=0,\\  \end{aligned}  \right.\\
		(c)&\left\{  \begin{aligned}
			\overset{.}\mu^{T}(t) &=-\frac{\partial
				H(t,x,y,K^*(\cdot),\lambda(t),\mu(t))}{\partial y},\\
			\mu^T(T)&=0,\\  \end{aligned}  \right.\\
		(d)&H(t,x,y,K^*(\cdot),\lambda(t),\mu(t))=\displaystyle\max_{K(\cdot)\in \Omega}
		H(t,x,y,K(\cdot),\lambda(t),\mu(t)).
	\end{aligned}
\end{equation}
\end{theorem}

\begin{proof}
Integrating by parts, the above formula can be transformed into
\begin{equation}
\begin{aligned}
	J_2 [K^*(\cdot)]=&\lambda^T (T)x(T)-\lambda^T (0) x_0+\mu^T (T)y(T)-\mu^T (0) y_0\\
	&-\int_{0}^{T}\{ \dot{\lambda}^T(t)x(t)+\dot{\mu}^T(t)y(t)\\
	&-\dfrac{1}{T}||K^*(x(t,x_0))-y(t,K^*(x_0))||^2+\lambda^T (t)f(t,x)+\mu^T
	(t)g(t,y) \}dt\\
	=&\lambda^T (T)x(T)-\lambda^T (0) x_0+\mu^T (T)y(T)-\mu^T (0) y_0\\
	&-\int_{0}^{T}\{
	\dot{\lambda}^T(t)x(t)+\dot{\mu}^T(t)y(t)+H(t,x,y,K^*(\cdot),\lambda(t),\mu(t))\}dt,
\end{aligned}
\end{equation}
where $H(t,x(t),y(t),K^*(\cdot),\lambda (t),\mu (t))$

$\quad\quad=-\dfrac{1}{T}||K^*(x(t,x_0))-y(t,K^*(x_0))||^2+\lambda^T (t)f(t,x)+\mu^T
(t)g(t,y)  . $\\

Obviously,
\begin{equation*}
    (a)\left\{ \begin{aligned}
		\overset{.}x(t) &=H_\lambda^T,  \\
		\overset{.}y(t) &=H_\mu^T. \\ \end{aligned}  \right. \\
\end{equation*}
Let $\epsilon$ be a sufficiently small real number, and $\delta K(\cdot)$ be
an arbitrarily determined piecewise continuous vector valued function on $[0,T]$.
Let $K^{\epsilon} =K^*+\epsilon \delta K$, according to the continuous dependence of the solution on the parameter, the solution of $K^{\epsilon} $ can correspondingly be expressed as
\begin{equation}\begin{aligned}
	x^{\epsilon}(t)&=x(t)+\epsilon\delta
	x(t)+\frac{\epsilon^2}{2}\delta^2x(t)+o(\epsilon^2),\\
		y^{\epsilon}(t)&=y(t)+\epsilon\delta
		y(t)+\frac{\epsilon^2}{2}\delta^2y(t)+o(\epsilon^2).
\end{aligned}
\end{equation}
Besides, $\dot{x}=f(t,x),~\dot{y}=g(t,y)$. Then
\begin{equation}
\begin{aligned}
	&\dot{x}^\epsilon=f(t,x^\epsilon),~\dot{y}^\epsilon=g(t,y^\epsilon),\\
	&\dot{x}^\epsilon-\dot{x}=\frac{dx^\epsilon(t)}{dt}-\frac{dx(t)}{dt}=
\frac{d}{dt}\delta x(t)\epsilon+o(\epsilon),\\
&\dot{y}^\epsilon-\dot{y}=\frac{dy^\epsilon(t)}{dt}-\frac{dy(t)}{dt}=
\frac{d}{dt}\delta y(t)\epsilon+o(\epsilon),\\
	&f(t,x^\epsilon)-f(t,x)=\frac{\partial f}{\partial x}(t,x)\delta
	x(t)\epsilon+o(\epsilon),\\
&g(t,y^\epsilon)-g(t,y)=\frac{\partial g}{\partial y}(t,y)\delta
	y(t)\epsilon+o(\epsilon).
\end{aligned}
\end{equation}
Comparing the coefficient of the first power of $\epsilon$, it can be obtained
that
\begin{equation}
\begin{aligned}
	&\left\{ \begin{aligned}
		\frac{d}{dt}\delta x(t)&=\frac{\partial f}{\partial x}(t,x)\delta x(t),
		\\
		\delta x(0)&=0, \\ \end{aligned}  \right.\\
	&\left\{  \begin{aligned}
		\frac{d}{dt}\delta y(t)&=\frac{\partial g}{\partial y}(t,y)\delta y(t),
		\\
		\delta y(0)&=0,\\  \end{aligned}  \right.\\
	&\delta^2x(t)=\displaystyle\lim_{\epsilon \rightarrow
	0}\frac{2}{\epsilon^2}[x^\epsilon (t)-x(t)+\epsilon \delta x(t)],\\
	&\delta^2y(t)=\displaystyle\lim_{\epsilon \rightarrow
	0}\frac{2}{\epsilon^2}[y^\epsilon (t)-y(t)+\epsilon \delta y(t)].
\end{aligned}
\end{equation}
Substitute $K^\epsilon,~x^\epsilon (t),~y^\epsilon (t)$ in $J_2 [\cdot]$:
\begin{equation}
\begin{aligned}
	J_2 [K^\epsilon(\cdot)]=&\lambda^T (T)[x(T)+\epsilon \delta
	x(T)+\frac{\epsilon^2}{2}\delta^2x(T)+o(\epsilon^2)]-\lambda^T (0) x_0\\
	&+\mu^T (T)[y(T)+\epsilon \delta
	y(T)+\frac{\epsilon^2}{2}\delta^2y(T)+o(\epsilon^2)]-\mu^T (0) y_0\\
	&-\int_{0}^{T}\dot{\lambda}^T(t)[x(t)+\epsilon \delta
	x(t)+\frac{\epsilon^2}{2}\delta^2x(t)+o(\epsilon^2)]dt\\
	&-\int_{0}^{T}\dot{\mu}^T(t)[y(t)+\epsilon \delta
	y(t)+\frac{\epsilon^2}{2}\delta^2y(t)+o(\epsilon^2)]dt\\
	&-\int_{0}^{T}H(t,x^\epsilon,y^\epsilon,K^\epsilon(\cdot),\lambda(t),\mu(t))dt,
\end{aligned}
\end{equation}
\begin{equation*}
\begin{aligned}
	\Delta J_2[K^*(\cdot)]=&J_2[K^\epsilon(\cdot)]-J_2[K^*(\cdot)]\\
	=&\lambda^T (T)[\epsilon \delta
	x(T)+\frac{\epsilon^2}{2}\delta^2x(T)+o(\epsilon^2)]+\mu^T (T)[\epsilon
	\delta y(T)+\frac{\epsilon^2}{2}\delta^2y(T)+o(\epsilon^2)]\\
	&-\int_{0}^{T}\dot{\lambda}^T(t)[\epsilon \delta
	x(t)+\frac{\epsilon^2}{2}\delta^2x(t)+o(\epsilon^2)]dt\\
    &-\int_{0}^{T}\dot{\mu}^T(t)[\epsilon \delta
	y(t)+\frac{\epsilon^2}{2}\delta^2y(t)+o(\epsilon^2)]dt\\
	&-\int_{0}^{T}\Delta H(t,x,y,K^*(\cdot),\lambda(t),\mu(t))dt,\\
    \text{where} \quad\Delta H=&H(t,x^\epsilon,y^\epsilon,K^\epsilon(\cdot),\lambda(t),
	\mu(t))-H(t,x,y,K^*(\cdot),\lambda(t),\mu(t))\\
	=&\epsilon\bigg[\frac{\partial H}{\partial x}\delta x(t)+\frac{\partial
	H}{\partial y}\delta y(t)+\frac{\partial H}{\partial K}\delta K(t)\bigg]\\
	&+\frac{\epsilon^2}{2}\Bigg\{\frac{\partial H}{\partial x}\delta^2
	x+\frac{\partial H}{\partial y}\delta^2 y+\frac{\partial H}{\partial
	K}\delta^2 K\\
&~~~~~~~~~~+2[\delta^Tx,\delta^Ty,\delta^TK]
\left[
\begin{matrix}
	\frac{\partial^2H}{\partial x^2}&\frac{\partial^2H}{\partial x\partial
	y}&\frac{\partial^2H}{\partial x\partial K}\\
	\frac{\partial^2H}{\partial y\partial x}&\frac{\partial^2H}{\partial
	y^2}&\frac{\partial^2H}{\partial y\partial K}\\
	\frac{\partial^2H}{\partial K\partial x}&\frac{\partial^2H}{\partial
	K\partial y}&\frac{\partial^2H}{\partial K^2}
\end{matrix}
\right]
\left[
\begin{matrix}
	\delta^Tx\\
	\delta^Ty\\
	\delta^TK
\end{matrix}
\right]
	\Bigg\}.
\end{aligned}
\end{equation*}

Moreover, $\Delta J_2 [K^*(\cdot)]=\epsilon \delta J_2 [K^*(\cdot)]+\frac{\epsilon^2}{2} \delta^2 J_2
	[K^*(\cdot)]+o(\epsilon^2 )$. Comparing the coefficient of the first power of $\epsilon$, we have
\begin{equation}
\begin{aligned}
	\delta J_2[K^*(\cdot)]=&\lambda^T(T)\delta x(T)+\mu^T(T)\delta y(T)\\
	&-\int_{0}^{T}[\dot{\lambda}^{T}(t)+\frac{\partial H}{\partial x}]\delta
	x(t)dt-\int_{0}^{T}[\dot{\mu}^{T}(t)+\frac{\partial H}{\partial y}]\delta y(t)dt-\int_{0}^{T}\frac{\partial H}{\partial K}\delta K(\cdot)dt.
\end{aligned}
\end{equation}
Let
\begin{equation*}\begin{aligned}
&(b)\left\{  \begin{aligned}
		\overset{.}\lambda^{T}(t) &=-\frac{\partial
			H(t,x,y,K^*(\cdot),\lambda(t),\mu(t))}{\partial x},\\
		\lambda^T(T)&=0,\\  \end{aligned}  \right.~\\
&(c)\left\{  \begin{aligned}
		\overset{.}\mu^{T}(t) &=-\frac{\partial
			H(t,x,y,K^*(\cdot),\lambda(t),\mu(t))}{\partial y},\\
		\mu^T(T)&=0.\\  \end{aligned}  \right.
\end{aligned}
\end{equation*}
It follows from $\delta J_2[K^*(\cdot)]=0$ that $\int_{0}^{T}\frac{\partial H}{\partial K}\delta K(\cdot)dt=0$. According to variational fundamental theorem (Lemma \ref{L2}),
\begin{equation}
	\frac{\partial
		H(t,x,y,K^*(\cdot),\lambda(t),\mu(t))}{\partial K}=0.
\end{equation}

Comparing the coefficient of the second power of $\epsilon$, it yields
\begin{equation}
\begin{aligned}
	\delta^2 J_2[K^*(\cdot)]=&\lambda^T(T)\delta^2 x(T)+\mu^T(T)\delta^2 y(T)\\
	&-\int_{0}^{T}[\dot{\lambda}^{T}(t)+\frac{\partial H}{\partial x}]\delta^2
	x(t)dt-\int_{0}^{T}[\dot{\mu}^{T}(t)+\frac{\partial H}{\partial y}]\delta^2 y(t)dt-\int_{0}^{T}\frac{\partial H}{\partial K}\delta^2 K(\cdot)dt\\
	&-2\int_{0}^{T}[\delta^Tx,\delta^Ty,\delta^TK]
	\left[
	\begin{matrix}
		\frac{\partial^2H}{\partial x^2}&\frac{\partial^2H}{\partial x\partial
			y}&\frac{\partial^2H}{\partial x\partial K}\\
		\frac{\partial^2H}{\partial y\partial x}&\frac{\partial^2H}{\partial
			y^2}&\frac{\partial^2H}{\partial y\partial K}\\
		\frac{\partial^2H}{\partial K\partial x}&\frac{\partial^2H}{\partial
			K\partial y}&\frac{\partial^2H}{\partial K^2}
	\end{matrix}
	\right]
	\left[
	\begin{matrix}
		\delta^Tx\\
		\delta^Ty\\
		\delta^TK
	\end{matrix}
	\right]
	dt.
\end{aligned}
\end{equation}
Note that $\delta^2 J_2 [K^*(\cdot)]\geqslant0$, then
\begin{equation}
	\frac{\partial^2
		H(t,x,y,K^*(\cdot),\lambda(t),\mu(t))}{\partial K^2}\leqslant0.
\end{equation}

That is, for all continuous times $t$ on $[0,T]$, there is
\begin{equation*}
	(d)H(t,x,y,K^*(\cdot),\lambda(t),\mu(t))=\displaystyle\max_{K(\cdot)\in \Omega}
	H(t,x,y,K(\cdot),\lambda(t),\mu(t)).
\end{equation*}
\end{proof}
\subsection{Examples}\label{egs}
\ \ \ \ 
We consider two examples, which help us to understand $K^*(\cdot)$.
\begin{example}\label{eg1}
Consider a constant linear system $x(t)$, and another system $y(t)$ whose derivative is obtained by multiplying the derivative of this system by a matrix,
\begin{equation}\label{linear-simple}
		\left\{ \begin{aligned}
			\overset{.}x(t) &=Ax+B,\\
			x(0)&=x_0, \\ \end{aligned}  \right.~
		\left\{  \begin{aligned}
			\overset{.}y(t) &=C\overset{.}x(t),\\
			y(0)&=y_0,\\  \end{aligned}  \right.
\end{equation}
where $A,~C$ are constant matrices of order $n$, and $B$ is an $n$-dimensional vector.
\end{example}

The solution of  (\ref{linear-simple}) is
\begin{equation}
\begin{aligned}
	x(t)&=e^{At} x_0+\int_{0}^{t}e^{A(t-s)}B ds,\\
	y(t)&=Cx(t)+D,\\
	D&=y_0-Cx_0.
\end{aligned}
\end{equation}
Substitute $x (t)$ to get
\begin{equation}
\begin{aligned}
y(t)&=Ce^{At} x_0+C\int_{0}^{t}e^{A(t-s)}B ds+D,\\
D&=y_0-Cx_0.
\end{aligned}
\end{equation}
Let $K(x(t,x_0 ))=y(t,K(x_0 )),$
\begin{equation*}
\begin{aligned}
&K(e^{At} x_0+\int_{0}^{t}e^{A(t-s)}B ds)=Ce^{At} x_0+C\int_{0}^{t}e^{A(t-s)} Bds+D,\\
&D=y_0-Cx_0=K(x_0 )-Cx_0.
\end{aligned}
\end{equation*}
In this instance, we can easily find $K^*(\cdot)$:
\begin{equation}
K^*(x)=Cx+y_0-Cx_0.
\end{equation}

Then $J[K^*]\equiv0$, $K^*(x(t,x_0))=y(t,K^*(x_0))$, $K^*(\cdot)$ is continuous and linear. When  $\det(C)\ne0$, $K^*(\cdot)$ is a homeomorphic map, and consequently $x(t)$ and $y(t)$ are conjugate. Thus two systems are completely similar. In this example, $K^*(\cdot)$ is a constant matrix of order $n$, and two systems are linearly similar.

Apply the maximum principle (Theorem \ref{max1}), $\lambda(t)=0~(b)$, $\mu(t)=0~(c)$, $J[K^*]\equiv0$, and the condition $(d)$ in (\ref{a1b1c1d1}) implies $$K^*=\arg \displaystyle\min_{K(\cdot)\in \Omega}J[K(\cdot)]=\arg \displaystyle\max_{K(\cdot)\in \Omega}
	H(t,x,y,K(\cdot),\lambda(t),\mu(t)).$$

\begin{example}\label{eg2}
Considere two linear systems:
\begin{equation}\label{linear-simple2}
		\left\{ \begin{aligned}
			\overset{.}x(t) &=Ax+B,\\
			x(0)&=x_0, \\ \end{aligned}  \right.~
		\left\{  \begin{aligned}
			\overset{.}y(t) &=Cy+D,\\
			y(0)&=y_0,\\  \end{aligned}  \right.
\end{equation}
where $A,~C$ are constant matrices of order $n$, and $B,~D$ are $n$-dimensional vectors.
\end{example}
Firstly, we consider the one-dimensional case, $n=1$, $x,~y\in \mathbb{R}$. The solution of
(\ref{linear-simple2}) are
\begin{equation}\label{linear solution}
\begin{aligned}
	x(t)&=e^{At}x_0+\int_{0}^{t}e^{A(t-s)}B ds\\
	&=e^{At} x_0+\frac{B}{A} (e^{At}-1),\\
	y(t)&=e^{Ct} y_0+\int_{0}^{t}e^{C(t-s)}D ds\\
	&=e^{Ct} y_0+\frac{D}{C} (e^{Ct}-1)\\
	&=\bigg(\frac{D}{C}+y_0 \bigg)
	\bigg(\frac{x(t)+\frac{B}{A}}{x_0+\frac{B}{A}}\bigg)^{\frac{C}{A}}-\frac{D}{C}  .
\end{aligned}
\end{equation}
According to
\begin{equation*}
\begin{aligned}
	\left\{ \begin{aligned}
		y_0&=K(x_0 ),\\
		y(t,K(x_0 ))&=K(x(t,x_0 )), \\ \end{aligned}  \right.
\end{aligned}
\end{equation*}
$K(x):~\mathbb{R}\rightarrow\mathbb{R}$ can be solved:
\begin{equation}
	K(x)=\bigg(\frac{D}{C}+y_0 \bigg)
	\bigg(\frac{x(t)+\frac{B}{A}}{x_0+\frac{B}{A}}\bigg)^{\frac{C}{A}}-\frac{D}{C}.
\end{equation}

For the two-dimensional case, $n=2$, $x,~y\in \mathbb{R}^2$. It should be
noted that the high-dimensional matrix cannot be divided, so the denominator needs to become the inverse of the multiplication matrix. Therefore, the condition also needs to ensure that matrices $A$ and $C$ are reversible (nonsingular matrix). In order to get the form of $K^*(\cdot)$, we put the solution $x(t)$ on the complex plane:
\begin{equation*}
x(t)+A^{-1}B=\alpha(t)e^{i\theta(t)},~
x_0+A^{-1}B=\alpha(0)e^{i\theta(0)},
\end{equation*}
where $\alpha(t)\in [0,T]\rightarrow\mathbb{R},~\theta(t) \in [0,T]\rightarrow[0,2\pi]$. Subsequently, we can obtain: 
\begin{equation}
K^*(x)=\bigg[\frac{x+A^{-1}B}{\alpha(0)}e^{-i\theta(0)}\bigg]^{A^{-1} C}(C^{-1} D+y_0 )-C^{-1}D.
\end{equation}
In fact,
\begin{equation*}
\begin{aligned}
	K(x)&=[\frac{x+A^{-1}B}{\alpha(0)}e^{-i\theta(0)}]^{A^{-1} C}(C^{-1} D+y_0 )-C^{-1}D\\
    &=[\frac{e^{At}(x_0+A^{-1}B)}{\alpha(0)}e^{-i\theta(0)}]^{A^{-1} C}(C^{-1} D+y_0 )-C^{-1}D\\
    &=[\frac{e^{At}\alpha(0)e^{i\theta(0)}}{\alpha(0)}e^{-i\theta(0)}]^{A^{-1} C}(C^{-1} D+y_0 )-C^{-1}D\\
    &=[e^{At}]^{A^{-1} C}(C^{-1} D+y_0 )-C^{-1}D\\
    &=[e^{Ct}](C^{-1} D+y_0 )-C^{-1}D\\
    &=e^{Ct}y_0+e^{Ct}C^{-1} D-C^{-1}D\\
    &=e^{Ct}y_0+(e^{Ct}-I)C^{-1}D\\
	&=y(t).
\end{aligned}
\end{equation*}

We get the form of $K^*(\cdot)$ when $n = 1, 2$. At this time, $J[K^*]\equiv 0,~K^*(x(t,x_0 ))=y(t,K^*(x_0))$. For the case of $n\geqslant3$, the form of $K^*(\cdot)$ is more complex. Anyway, logarithmic functions are not monomorphic, $K^*(\cdot)$ is not a homeomorphism, so $x(t)$ and $y(t)$ cannot conjugate in general. In this example, in order to force the conjugate condition to be satisfied, the obtained $K^*(\cdot)$ may not be homeomorphic.\\

Example \ref{eg2} shows that even the two most basic linear systems can not meet the conjugate property, which means that the conjugate property is a very strong one.

In order to obtain the conditions that homeomorphism $K^*$ satisfies, we substitute (\ref{linear-simple2}) as an example of (\ref{2nonlinear}) into the maximum principle (Theorem \ref{max1}):
\begin{equation}
\begin{aligned}
	(a')&\left\{ \begin{aligned}
		\overset{.}x(t) &=H_\lambda^T,  \\
		\overset{.}y(t) &=H_\mu^T ,\\ \end{aligned}  \right. \\
	(b')&\left\{  \begin{aligned}
		\overset{.}\lambda(t)
		&=-A^T\lambda(t)+\frac{2}{T}K^*(K^*(x(t,x_0))-y(t,K^*(x_0))),\\
		\lambda(T)&=0,\\  \end{aligned}  \right.\\
	(c')&\left\{  \begin{aligned}
		\overset{.}\mu(t) &=-C^T\mu(t)-\frac{2}{T}(K^*(x(t,x_0))-y(t,K^*(x_0))),\\
		\mu(T)&=0,\\  \end{aligned}  \right.\\
	(d')&H(t,x,y,K^*(\cdot),\lambda(t),\mu(t))=\displaystyle\max_{K(\cdot)\in\Omega}
	H(t,x,y,K(\cdot),\lambda(t),\mu(t)),
\end{aligned}
\end{equation}
where
$H(t,x(t),y(t),K(\cdot),\lambda (t),\mu (t))$
	
	$\quad\quad=-\dfrac{1}{T}||K(x(t,x_0))-y(t,K(x_0))||^2+\lambda^T
	(t)(Ax+B)+\mu^T (t)(Cy+D) . $

Starting from $(b')$ and $(c')$ , we can get that all above $\lambda (t),~\mu (t)$ have the following general solution:
\begin{equation}\begin{aligned}
	\lambda(t)&=\int_{T}^{t}\Phi(t)\Phi^{-1}(\tau)\frac{2}{T}K^*(K^*(x(\tau,x_0))-y(\tau,K^*(x_0)))d\tau,\\
	\mu(t)&=\int_{T}^{t}\Psi(t)\Psi^{-1}(\tau)(-\frac{2}{T}(K^*(x(\tau,x_0))-y(\tau,K^*(x_0))))d\tau,
\end{aligned}\end{equation}
where $\Phi(t)$ is the fundamental solution matrix of $\overset{.}\lambda(t)
		=-A^T\lambda(t)$,
$\Psi(t)$ is the fundamental solution matrix of $\overset{.}\mu(t) =-C^T\mu(t).$

From Theorem \ref{max1} $(d)$, we know that the minimizer $K^*$ such that
$$
	\frac{\partial
		H(t,x,y,K^*(\cdot),\lambda(t),\mu(t))}{\partial K}=0,
$$
which is equivalent to
\begin{equation}
	-\frac{2}{T}(K^*(x(t,x_0))-y(t,K^*(x_0)))^T(x(t,x_0)-\frac{\partial y(t,K^*(x_0))}{\partial y_0}x_0)=0.
\end{equation}
Substituting the solution (\ref{linear solution}) of (\ref{linear-simple2}) yields
\begin{equation}
	((e^{At}-e^{Ct})K^*(x_0)-\int_{0}^{t}K^*(e^{A(t-s)}B)-e^{C(t-s)}Dds)^T((e^{At}-e^{Ct})x_0+
\int_{0}^{t}e^{A(t-s)}Bds)=0.
\end{equation}

\section{Uniqueness and Similarity-based on Embedology}\label{Uniqueness Similarity}
\subsection{Uniqueness}\label{Uniqueness}
\ \ \ \ In Section \ref{Sufficient} and Section \ref{MAX}, we have discussed the
existence of $K$. Now we will observe if it is unique. The existence of $K$ in Proposition \ref{Hartman} is derived from the Banach contraction mapping principle and $K$ is unique.
In this section, we consider the uniqueness of $K(t)$ in Theorem \ref{exists in GL(n)}.\\

As we all know, the orbits of autonomous systems do not intersect. Since $K(t)$ has $n^2$ bases, $n^2$ constraints are needed to ensure uniqueness. For non-autonomous systems, we usually consider $f(t,x),~g(t,y)$ that are piece-wise constant, i.e.
\begin{equation}	
	\left\{ \begin{aligned}
		&\overset{.}{\hat{x}}(t)=\hat{f}(\hat{x}(t)),\\
		          &\hat{x}=(x,t),\hat{f}=(f,1) ,\\ \end{aligned}  \right.~
	\left\{  \begin{aligned}
		&\overset{.}{\hat{y}}(t)=\hat{g}(\hat{y}(t)),\\
	             &\hat{y}=(y,t),\hat{g}=(g,1) .\\  \end{aligned}  \right.
\end{equation}
In this case, $\hat{f}(\hat{x}),~\hat{g}(\hat{y})$ do not change explicitly with $t$, it is enough to consider the time-homogeneous equation, i.e. autonomous systems.\\

From Section \ref{MAX} Theorem \ref{max1}, we know that the minimizer $K^*(t)$ such that
$$
	\frac{\partial
		H(t,x,y,K^*(t),\lambda(t),\mu(t))}{\partial K}=0,
$$
which is equivalent to
\begin{equation}\label{KKT}
	\frac{\partial}{\partial k_{ij}(t)}(x^T(t)K^{*T}(t)K^*(t)x(t))
-2\frac{\partial}{\partial k_{ij}(t)}(x^T(t)K^{*T}(t)y(t))
+\frac{\partial}{\partial k_{ij}(t)}(y^T(t)y(t))=0,
\end{equation}
where $k_{ij}(t)$ denotes the element of the $i$-th row and $j$-th
column of the matrix $K^*(t)$, and $i,j=1,2,\cdots,n$.
Since the $n^2$ constraints in (\ref{KKT}) are irrelevant, $K(t)$ is unique and continuous with respect to $t$. (\ref{KKT}) is the Karush-Kuhn-Tucker (KKT) optimality condition.

Let $\bigg(M_{n\times n}\bigg)_{m,i}$ denote the element of the $m$-th row and $i$-th column of the matrix $M_{n\times n}$ and $\mathbf{e}_j(t)\triangleq (0,\cdots,0,x_j(t),0,\cdots,0)^T$.
For the first term in (\ref{KKT}), we obtain the following equation
by means of the chain rule of derivatives:
$$\begin{aligned}
\frac{\partial}{\partial k_{ij}(t)}(x^T(t)K^{*T}(t)K^*(t)x(t))&=
x^T(t)\frac{\partial K^{*T}(t)}{\partial k_{ij}(t)}K^*(t)x(t)+
x^T(t)K^{*T}(t)\frac{\partial K^*(t)}{\partial k_{ij}(t)}x(t)\\
&=\mathbf{e}_j^T(t)K^*(t)x(t)+x^T(t)K^{*T}(t)\mathbf{e}_j(t)\\
&=2\underset{l=1}{\overset{n}{\sum}}x_j(t)k_{jl}(t)x_l(t).
\end{aligned}
$$

For the second term in (\ref{KKT}), we have
$$\begin{aligned}
\frac{\partial}{\partial k_{ij}(t)}(x^T(t)K^{*T}(t)y(t))&=
x^T(t)\frac{\partial K^{*T}(t)}{\partial k_{ij}(t)}y(t)+
x^T(t)K^{*T}(t)\frac{\partial y(t)}{\partial k_{ij}(t)}\\
&=\mathbf{e}_j^T(t)y(t)+
x^T(t)K^{*T}(t)\frac{\partial y(t)}{\partial y_0(t)}\frac{\partial y_0(t)}{\partial k_{ij}(t)}\\
&=x_j(t)y_j(t)+x^T(t)K^{*T}(t)\frac{\partial y(t)}{\partial y_0(t)}\mathbf{e}_j(0)\\
&=x_j(t)y_j(t)+\underset{l=1}{\overset{n}{\sum}}\underset{m=1}{\overset{n}{\sum}}
x_j(0)k_{lm}(t)x_l(t)\bigg(\frac{\partial y(t)}{\partial y_0(t)}\bigg)_{m,i}.
\end{aligned}
$$

For the last term in (\ref{KKT}), it is enough to compute that
$$\begin{aligned}
\frac{\partial}{\partial k_{ij}(t)}(y^T(t)y(t))&=
\frac{\partial y^T(t)}{\partial k_{ij}(t)}y(t)+y^T(t)\frac{\partial y(t)}{\partial k_{ij}(t)}\\
&=\mathbf{e}_j^T(0)\bigg(\frac{\partial y(t)}{\partial y_0(t)}\bigg)^Ty(t)
+y^T(t)\frac{\partial y(t)}{\partial y_0(t)}\mathbf{e}_j(0)\\
&=2\underset{l=1}{\overset{n}{\sum}}x_j(0)y_l(t)\bigg(\frac{\partial y(t)}{\partial y_0(t)}\bigg)_{l,i}.
\end{aligned}
$$
Substitute the three terms into formula (\ref{KKT}), and we get
\begin{equation}\label{KKT1}\begin{aligned}
2\underset{l=1}{\overset{n}{\sum}}x_j(t)k_{jl}(t)x_l(t)
-2\bigg[x_j(t)y_j(t)&+\underset{l=1}{\overset{n}{\sum}}\underset{m=1}{\overset{n}{\sum}}
x_j(0)k_{lm}(t)x_l(t)\bigg(\frac{\partial y(t)}{\partial y_0(t)}\bigg)_{m,i}\bigg]\\
&+2\underset{l=1}{\overset{n}{\sum}}x_j(0)y_l(t)\bigg(\frac{\partial y(t)}{\partial y_0(t)}\bigg)_{l,i}=0.
\end{aligned}
\end{equation}

Consider $x(t)$, $y(t)$ satisfy (\ref{Hartman y}) and (\ref{Hartman x}) in the Hartman-Grobman theorem,
substitute $\frac{\partial y(t)}{\partial y_0(t)}=e^{At}$ into formula (\ref{KKT1}), and we get
\begin{equation}\label{KKT2}\begin{aligned}
2\underset{l=1}{\overset{n}{\sum}}x_j(t)k_{jl}(t)x_l(t)
-2\bigg[x_j(t)y_j(t)&+\underset{l=1}{\overset{n}{\sum}}\underset{m=1}{\overset{n}{\sum}}
x_j(0)k_{lm}(t)x_l(t)\bigg(e^{At}\bigg)_{m,i}\bigg]\\
&+2\underset{l=1}{\overset{n}{\sum}}x_j(0)y_l(t)\bigg(e^{At}\bigg)_{l,i}=0.
\end{aligned}
\end{equation}

Note that $i,j=1,\cdots,n$ in (\ref{KKT}), (\ref{KKT1}) and (\ref{KKT2}), so they contain $n^2$ constraints, and the solved $K^*(t)$ is unique.
\subsection{Similarity degree}
\ \ \ \ After solving \textbf{the minimizer $K(t)$}, we can define the
similarity between two systems. Obviously, if two systems are conjugate, the denominator in the above definition is 0. According to Definition \ref{similarity0}, the similarity degree between two systems is 1 ($100\%$).

For further calculation (Section \ref{test}), we discretize the trajectory and give the discrete
similarity definition:
\begin{definition}\textbf{(Discrete similarity degree)}\label{similarity2}
\begin{equation*}
	\rho(K(\cdot))=\left\{ \begin{aligned}
		\dfrac{log\Big(1+\dfrac{1}{N}
	 \underset{i=1}{\overset{N}{\sum}}||K(i)x(i,x_0)-y(i,y_0)||^2\Big)}{\dfrac{1}{N} \underset{i=1}{\overset{N}{\sum}}||K(i)x(i,x_0)-y(i,y_0)||^2 }, &\qquad \underset{i=1}{\overset{N}{\sum}}||K(i)x(i,x_0)-y(i,y_0)|| \ne 0,\\
		          1,\qquad\qquad &\qquad \underset{i=1}{\overset{N}{\sum}}||K(i)x(i,x_0)-y(i,y_0)|| = 0, \end{aligned}  \right.
\end{equation*}
where $K(i)=\{k_{jk}(i)\}_{j,k=1}^n$.
\end{definition}

Analogous to Section \ref{The main idea} \textbf{(v)}, if the minimizer $K(i)$ is a constant matrix, then the similarity is the linear similarity. If $K(i)$ is an orthogonal matrix (metric preserving) or symplectic matrix (differential structure preserving), the similarity is the rigid similarity.

\subsection{Similarity theorem based on Embedology}\label{embed}
\ \ \ \ Takens embedding theorem is a fundamental principle to calculate the dimension of embedded phase space \cite{[Saue+91]}. An important problem in applications of chaos is to reconstruct an $n$-dimensional phase space that can contain the chaotic motion from the time series of a single variable.

For the question of how big the dimension $n$ should be, in 1934, American mathematician Whitney and in 1980, Takens \cite{[Takens+81]} successively proved the embedding theorem of the required dimension $n$: in order to ensure that the phase space can accommodate the topological characteristics of the original attractor of the state space, if the original attractor is in a $ d$-dimensional space, the dimension of the phase space in which the attractor is embedded must reach $n=2d+1$.
For example, the dimension of the attractor obtained by reconstructing the phase space from a time series is $d=2. 51$, that is, the attractor is in a space with $d=3$ dimension, so the dimension of the embedded phase space is at least $n= 2\times3+1=7$.
\begin{remark}
It is noteworthy that $n^2>2n+1$, when $n\geq3,n\in\N^*$. This means that the reconstructed $n$-dimensional phase space is not unique by Takens embedding theorem.
\end{remark}

Back to the original systems (\ref{2nonlinear}), we only need to consider autonomous systems by
using the treatment of non-autonomous systems in Section \ref{Uniqueness},
\begin{equation}
	\left\{ \begin{aligned}
		\overset{.}x(t)&=f(x),\\
		          x(0)&=x_0 ,\\ \end{aligned}  \right.~
	\left\{  \begin{aligned}
		\overset{.}y(t)&=g(y),\\
	             y(0)&=y_0=K(0)x_0 .\\  \end{aligned}  \right.
\end{equation}

If we know the behavior of the system in $[0, T]$, combined with the Takens embedding theorem \cite{[Takens+81]} and the similarity theory, we can realize the reconstruction and prediction of the $n$-dimensional system. This section describes how $K(t)$ predicts the system completely and accurately, and how $n$-order constant matrix $K$ predicts the past and future of the system, corresponding to Algorithm \ref{algorithm1} and Algorithm \ref{algorithm2} of Section \ref{3d chaotic systems} respectively. \\

\textbf{Completely accurate prediction-$K(t)$.}

We use $y(t)$ to predict $x(t)$, in other words, solving $K(t)$ makes the similarity of the two systems reach 100$\%$ (completely accurate prediction). $K(t)$ needs to meet $n^2$ constraints as $K(t)\in GL(n)$. In fact, each point in $\R^n$ contains $n$ one-dimensional constraints, and $K(t)$ has $n^2$ one-dimensional constraints. Let
\begin{equation}
K(t)(x(t),c_1(t),\cdots ,c_{n-1}(t))=(y(t,K(0)x_0),y_1(t),\cdots ,y_{n-1}(t)),
\end{equation}
where $y_i(t)$ is the trajectories of the known $y$ system starting from different initial values; $\dot{c_i}=f_i(t,c_i)=f(c_i)+\epsilon_i\backsimeq f(c_i)$, $\backsimeq$ indicates approximate equality, $\epsilon_i$ is the random small disturbance; and $i=1,\cdots ,n-1$. Since the orbits of autonomous systems are disjoint, $c_i$ can be understood as $n-1$ approximate solution trajectories of the system $x$, or input data.

On the basis of the above analysis, $K(t)$ exists and is unique. At this time, the similarity between the two systems reaches $100\%$. We call it completely accurate prediction.\\

\begin{figure}[htbp]
\centering
\includegraphics[width=0.9\linewidth]{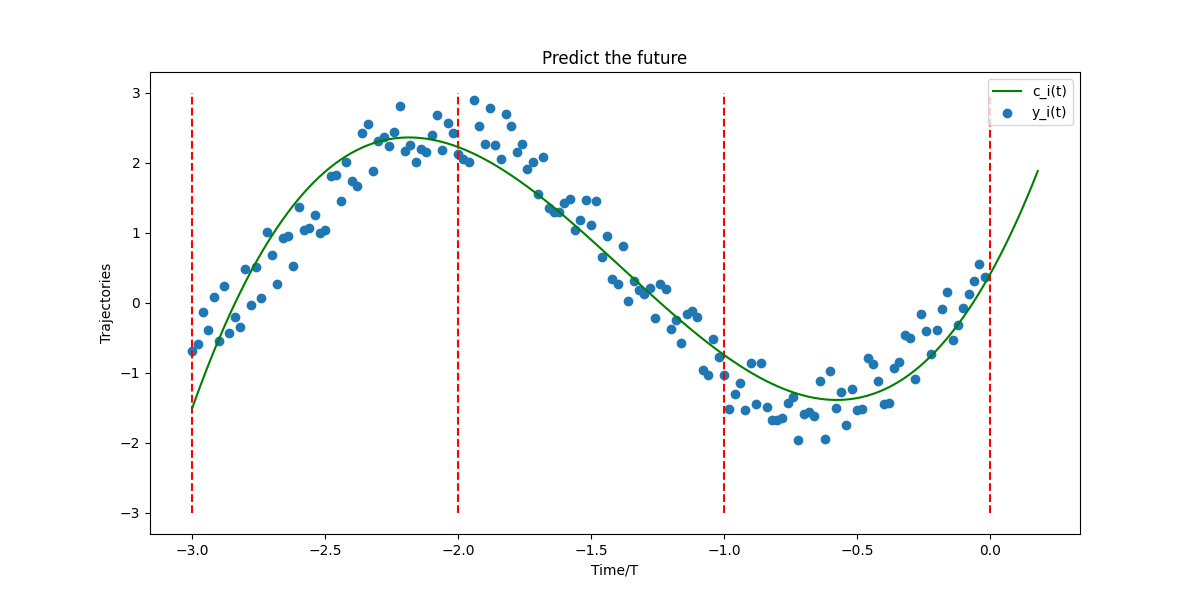}
\caption{\small{Forecast future diagram.}}
\label{ci}
{\small{The blue points represents $y_i(t)~(i=1,2,3)$, the trajectory of the known $x$ system on $[-3T,-2T],[-2T,-T],[-T,0]$. The green line indicates $c_i(t)~(i=1,2,3)$, which is a physical law learned from the past $n$-segment trajectory and used to predict the future.}}
\end{figure}

\textbf{Predict the future and infer the past-$K$.}

To predict the future from the past $n$ time behaviors of the system, or to infer the past from the future $n$ time behaviors, we can also adopt a similar method to solve $K\in \M^n$ to maximize the similarity between the two systems. Let
\begin{equation}
K(c_0(t),\cdots ,c_{n-1}(t))=(y_1(t),\cdots ,y_n(t)),
\end{equation}
where $y_i(t)$ is the trajectory of the known $x$ system on $[-(n-i+1)T,-(n-i)T],~i=1,\cdots ,n$; $y_{i}(-(n-i)T)=y_{i+1}(-(n-i)T)$, $i=1,\cdots ,n-1$; and
\begin{equation}
\dot{c_i}=\left\{
\begin{aligned}f_i(t,c_i)=f(x)+\epsilon_i\backsimeq f(x),& \quad [-(n-i)T,-(n-i-1)T],\\
 0, &\quad  \text{others else}. \\
 \end{aligned}\right.
 \end{equation}

The above analysis shows that $K$ is unique and reversible. Then our predicted $x_n$ is $x_n(0)=K^{-1}y_n(0)$. Taking $n = 3$ as an example, Figure \ref{ci} vividly demonstrates this process.

Analogically, let $y_i(t)$ is the trajectory of the known $x$ system on $[(n-i)T,(n-i+1)T],~i=1,\cdots ,n$; $y_{i}((n-i)T)=y_{i+1}((n-i)T)$, $i=1,\cdots ,n-1$; and
\begin{equation}
\dot{c_i}=\left\{
\begin{aligned}f_i(t,c_i)=f(x)+\epsilon_i\backsimeq f(x),&\quad  [(n-i-1)T,(n-i)T],\\
 0, &\quad  \text{others else}. \\
 \end{aligned}\right.
 \end{equation}
Then our inferred $x_{-n}$ is $x_{-n}(0)=K^{-1}y_n(0)$.
\section{Chaotic systems and Numerical simulation}\label{test}

\subsection{3d chaotic systems}\label{3d chaotic systems}
\ \ \ \ Three three-dimensional systems, Lorenz system, Chua's circuit system and
Chen's system, are considered in this section. The parameters of these systems
are derived from the literature (see \cite{[AFM+13],[CHEN+14]}). The reason why such
systems are considered is that they have attracted extensive attention.

Consider Lorenz system:
\begin{equation}
	\dot{X}=F(t,X)=\left( \begin{aligned}
		\overset{.}x_1(t)\\
		\overset{.}x_2(t)\\
		\overset{.}x_3(t) \end{aligned}  \right)=
\left(
	\begin{matrix}
	    \sigma(x_2-x_1)\\ \rho x_1-x_2-x_1x_3\\x_1x_2-\beta x_3
	\end{matrix}
\right)	,
\end{equation}
where $\sigma,\rho,\beta$ are positive real parameters. In the following numerical calculation, we adopt $X_0=(0.0,1.0,0.0),~\sigma_1=10,~\rho_1=28,~\beta_1=8/3$, and $\sigma_2=10,~\rho_2=28,~\beta_2=3$, respectively.

Consider Chua's circuit system:
\begin{equation}
	\dot{Y}=G(t,Y)=\left( \begin{aligned}
		\overset{.}y_1(t)\\
		\overset{.}y_2(t)\\
		\overset{.}y_3(t) \end{aligned}  \right)=
	\left(
	\begin{matrix}
		\alpha(y_2-y_1)-f(y_1)\\x_1-x_2+x_3\\x_1x_2-\beta y_2
	\end{matrix}
	\right)	,
\end{equation}
where $f(y_1)=m_1y_1+\frac{1}{2}(m_0-m_1)(|y_1+1|-|y_1-1|)$, and
$\alpha,\beta,m_0,m_1$ are real parameters. In the following numerical calculation, we adopt $Y_0=(0.1,0.3,-0.6),~\alpha=10,~\beta=15,~m_0=-1.2,~m_1=-0.6.$

Consider Chen's system:
\begin{equation}
	\dot{Z}=E(t,Z)=\left( \begin{aligned}
		\overset{.}z_1(t)\\
		\overset{.}z_2(t)\\
		\overset{.}z_3(t) \end{aligned}  \right)=
	\left(
	\begin{matrix}
		a(z_2-z_1)\\(c-a)z_1+cz_2-z_1z_3\\z_1z_2-bz_3
	\end{matrix}
	\right)	,
\end{equation}
where $a,b,c$ are positive real parameters. In the following numerical calculation, we adopt $Z_0=(0.0,1.0,0.0),~a=35,~b=3,~c=28.$\\

Next, we give the algorithm to solve $K (t)$, and give the numerical solution and the comparison of similarity.

Considerations about the similarity give rise to an algorithm for solving $K(t)$ (Algorithm \ref{algorithm1}) according to the maximum principle in Section \ref{MAX}. Subsequently, the phase diagrams (Figure \ref{Figure1}) of the four systems, the linear transformation $K$ of two different systems, and the $K(t)$ (Algorithm \ref{algorithm2}) with $100\%$ similarity of two different systems are obtained.\\

Note that we adopt the concept of time average (Definitions \ref{Cost functional}, \ref{similarity}, \ref{similarity2}), which means that the similarity will be higher and higher with the evolution of time, although there may be a period of similarity and a period of dissimilarity at the beginning.

For example, if we use $T = 300$, the final similarity is $94.46 \%$ , $98.99 \%$  and $90.55 \%$. This is illustrated in Figure \ref{Figure5}.  It can be seen that with the passage of time, the similarity also
increases, and finally tends to a stable state.

\begin{figure}[htbp]
\centering
\subfigure{
\begin{minipage}[b]{0.5\linewidth}
		\includegraphics[width=1\linewidth]{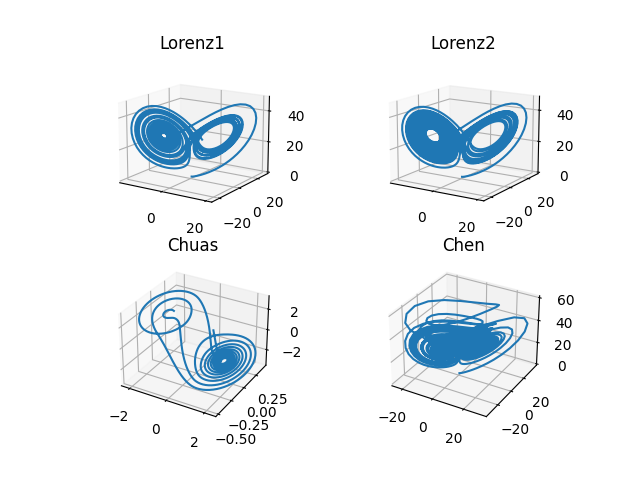}
		\caption*{\small{\textbf{a.} Trajectories of four nonlinear systems.}}
		\label{Figure1}
	\end{minipage}}
\subfigure{
\begin{minipage}[b]{0.45\linewidth}
			\includegraphics[width=1\linewidth]{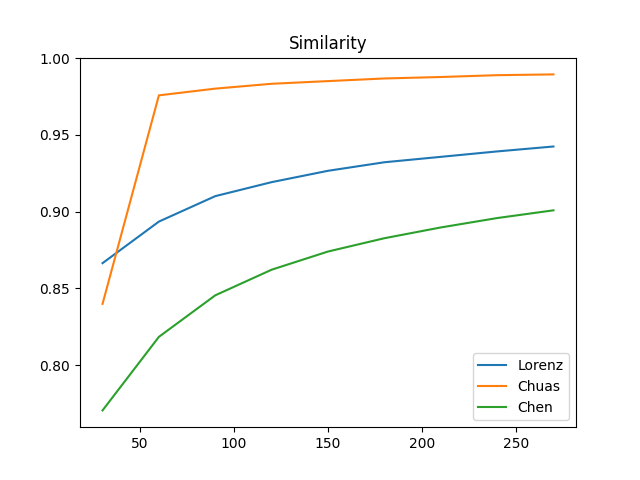}
			\caption*{\small{\textbf{b.} Similarity evolution graph.}}
			\label{Figure5}
	\end{minipage}}
\caption{\small{Comparison between Lorenz system and other three nonlinear systems.}}
\end{figure}

\begin{algorithm}
\caption{ Solving $K (t)$ with 100$\%$ similarity}
\label{algorithm1}

Step 1. Give $X_0,~Y_0$, solves two ODEs to obtain $X(t),~Y(t)$.

Step 2. Enter the cycle from step $i = 0$, and $N = T$ terminates the cycle.

Step 3. Takes columns $i$ to $i + n$ of $X(t),~Y(t)$.

Step 4. Calculate $K(i)$ by $K(i)X(i)=Y(i)$.

Step 5. $K(i)$ corresponds to the $i$-th matrix of $K(t)$.

Step 6. Let $i=i+1$, repeat Step 3 to Step 5 until the end.
\end{algorithm}
\begin{algorithm}
\caption{ Constant matrix $K$ for maximum similarity}
\label{algorithm2}

Step 1. Give $X_0,~Y_0$, solves two ODEs to obtain $X(t),~Y(t)$.

Step 2. The initial similarity is set to 0.

Step 3. Enter the cycle from step $i = 0$, and $N = T$ terminates the cycle.

Step 4. Takes columns $i$ to $i + n$ of $X(t),~Y(t)$.

Step 5. Calculate $K(i)$ by $K(i)X(i)=Y(i)$.

Step 6. Calculate the similarity (Definition \ref{similarity2}) of this step.

step 7. Compare the similarity with the previous step, the greater similarity and the corresponding matrix $K(i)$ are stored in this step.

Step 8. Let $i=i+1$, repeat step 4 to step 7 until the end.
\end{algorithm}

\subsection{Numerical solution}
\ \ \ \ Here, we study the similarity between the two systems through specific numerical simulation, it is divided into two steps:\\

\textbf{(i) Find the solution trajectory.}\\

Using differential equation solvers (there are already many differential
equation solvers to make this process easy to implement, such as Python $scipy.integrate. odeint$) to solve he trajectories of $X_1 (t),~X_2
(t),~Y(t),~Z(t)$, which are recorded as $track(i),~i=1,2,3,4$.\\

\textbf{(ii) Algorithms.}\\

Use Algorithms \ref{algorithm1} and \ref{algorithm2} to obtain the corresponding numerical solutions and similarity. Here, we adopt $T= 30, dt = 0.01$. For specific procedure details, see Python file. Figure \ref{Figure1} shows the trajectories of four systems. We can
clearly see that four nonlinear systems are the structure of double scroll chaotic
attractor, which has a certain similarity. The numerical solution
of Algorithm \ref{algorithm2} is shown in Table \ref{Table1}. The pairwise comparison of
four systems can  be seen intuitively in Figure [\ref{Figure 2},\ref{Figure 3},\ref{Figure 4}].\\

The following numerical results prove once again that two systems can reach the maximum value of linear similarity (under the minimizer constant matrix $K$), and the similarity increases with time; Under the time-varying $K(t)$, the two systems can be completely similar (the similarity reaches 100$\%$), i.e. the functional attains the extreme value of 0.
\begin{center}
	\begin{threeparttable}
   \captionof{table}{Similarity and constant matrix $K$ of different systems.} 
   \label{Table1}
		\begin{tabular}{l c c c }	
			\hline
			   Systems&Initial&Constant matrix K&Final\\
			   &similarity&   &similarity\\
			\hline				
			Lorenz1$\&$Lorenz2&18.35$\%$&$\(
			\begin{matrix}
				0.99669589&0.00708679&-0.00368821\\
				-0.01092915&1.02344135&-0.0121997\\
				-0.02563355&0.0549802&0.97138635
			\end{matrix}\)
		$&86.64$\%$ \\
			Lorenz$\&$Chua's&11.93$\%$&$\(
			\begin{matrix}
				2.06541069&-1.78187423&-0.16405216\\
				2.61340569&-2.59013073&0.11312293\\
				3.16451651&-3.40502316&0.39380414
			\end{matrix}\)
			$&83.99$\%$ \\
		    Lorenz$\&$Chen's&16.06$\%$&$\(
		    \begin{matrix}
		    	2.98413841&-2.96556925&0.57976059\\
		    	2.7629626&-2.65449576&0.49652093\\
		    	2.7629626&-2.65449576&0.49652093
		    \end{matrix}\)
		    $&77.05$\%$ \\
		    \hline
		\end{tabular}
				\begin{tablenotes}
					\item[] \small{Note: see Python file for detailed
					procedures.	This is the result at T = 30.}	
				\end{tablenotes}
	\end{threeparttable}
\end{center}

Compare the data in Table \ref{Table1} horizontally. The calculated minimizer $K$
can
significantly improve the similarity between two systems, reflecting the
similar geometric structure to a certain extent.

Compare the data in Table \ref{Table1} vertically. From the initial similarity,
two
Lorenz systems with different parameters are the most similar, followed by
Lorenz and Chen's, and Lorenz and Chua's are the
most unlike. From the similarity after multiplying $K$, two Lorenz systems are still the most similar, but the similarity
between Lorenz and Chua's increases the most, indicating that
Chua's can be very similar to Lorenz after linear
transformation.

\begin{figure}[htbp]
\centering
\subfigure{
\begin{minipage}[b]{0.33\linewidth}
			\includegraphics[width=1\linewidth]{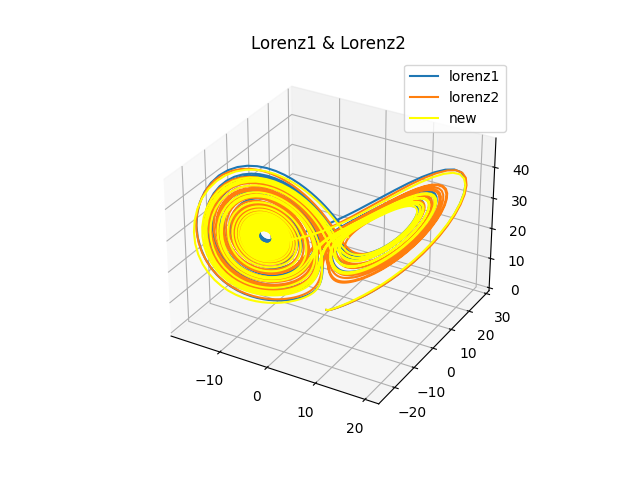}
			\caption*{\small{\textbf{a.} Two Lorenz systems.}}
			\label{Figure 2}
	\end{minipage}}
\subfigure{
\begin{minipage}[b]{0.3\linewidth}
		\includegraphics[width=1\linewidth]{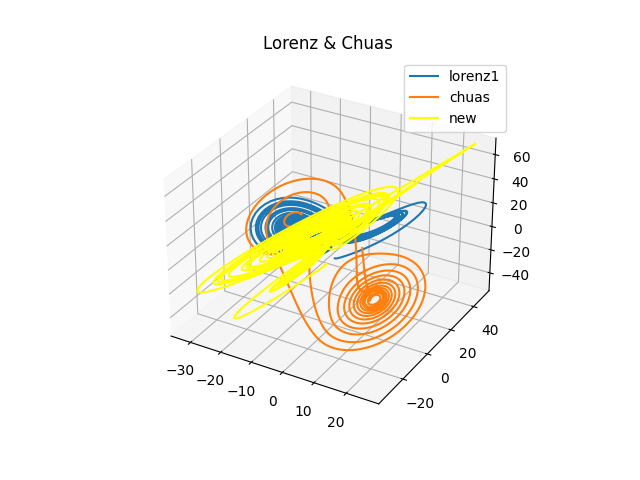}
		\caption*{\small{\textbf{b.} Lorenz system and Chua's circuit system.}}
		\label{Figure 3}
	\end{minipage}}
\subfigure{
\begin{minipage}[b]{0.33\linewidth}
		\includegraphics[width=1\linewidth]{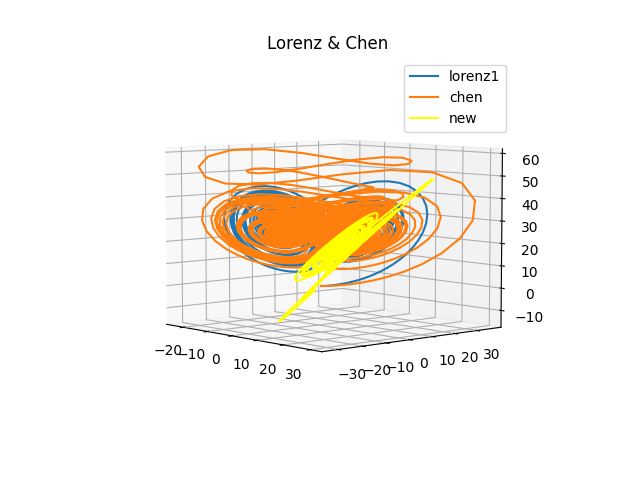}
		\caption*{\small{\textbf{c.} Lorenz system and Chen's system.}}
		\label{Figure 4}
	\end{minipage}}
\caption{\small{Pairwise comparison of four systems.}}
{\small{The blue line represents the trajectory of Lorenz system, orange represents the other system (Lorenz2, Chua's circuit, Chen's), and yellow represents the trajectory after $K$ transformation.}}
\end{figure}

This paper discusses the similarity between two deterministic ordinary differential systems. We will naturally ask, if the two dynamic systems are not ordinary differential equations (ODEs), but partial differential equations (PDEs) or stochastic differential equations (SDEs) with random terms, what is the similarity between them? After Pontryagin maximum principle, Peng \cite{[P+90]} proposed the stochastic maximum principle, and there have been some work, such as \cite{[HPW+10]} and \cite{[CPT+16]}. Can the maximum principle be considered as a necessary condition for the existence of similarity? What is the sufficient condition? Can we predict PDEs or SDEs based on the Takens embedding theorem? There are still many open problems. In the forthcoming paper, we will consider the similarity between two infinite dimensional dynamic systems or stochastic dynamic systems in detail.

 \end{document}